\documentclass[11pt]{article}

\usepackage{amsmath, amsfonts, amsthm, graphicx,amssymb}
\usepackage[top=21mm, bottom=22mm, left=21mm, right=21mm]{geometry}
\usepackage{comment}
\usepackage{hyperref}
\hypersetup{
	colorlinks=true,
	linkcolor=blue,
	filecolor=magenta,      
	urlcolor=cyan,
	citecolor=blue
}
\usepackage{authblk}
\usepackage{enumitem}
\usepackage[none]{hyphenat}
\usepackage{tikz, tikz-3dplot}
\usetikzlibrary{calc, patterns,intersections, scopes}
\usetikzlibrary{positioning,arrows,shapes,decorations.markings,decorations.pathreplacing, decorations.pathmorphing,matrix,patterns}
\tikzstyle{vtx}=[circle,draw=black,fill=black,inner sep=0,minimum size=2pt,text=white,font=\footnotesize]
\usetikzlibrary{backgrounds}
\pgfdeclarelayer{bg}
\pgfsetlayers{bg,main}
\usepackage{algorithm}
\usepackage{algpseudocode}
\algrenewcommand\algorithmicrequire{\textbf{Input:}}
\algrenewcommand\algorithmicensure{\textbf{Output:}}

\newtheorem{theorem}{Theorem}[section]
\newtheorem{corollary}[theorem]{Corollary}
\newtheorem{lemma}[theorem]{Lemma}
\newtheorem{claim}[theorem]{Claim}
\newtheorem{conjecture}[theorem]{Conjecture}

\usepackage{cleveref}
\crefname{lemma}{lemma}{lemmas}
\Crefname{lemma}{Lemma}{Lemmas}
\crefname{claim}{claim}{claims}
\Crefname{claim}{Claim}{Claims}
\crefname{question}{question}{questions}
\Crefname{question}{Question}{Questions}
\crefname{conjecture}{conjecture}{conjecture}
\Crefname{conjecture}{Conjecture}{Conjecture}

\theoremstyle{definition}

\newcommand{\eps}{\varepsilon}

\DeclareMathOperator{\bx}{box}

\DeclareMathOperator{\Var}{Var}

\title{Geometric triangle-free graphs of large chromatic number}
\author{Istv\'an Tomon}
\affil{Department of Mathematics and Mathematical Statistics\\ 
Ume\r{a} University, Ume\r{a}, Sweden}
\affil{\texttt{istvan.tomon@umu.se}}
\date{}

\begin{document}

\maketitle
\sloppy

\begin{abstract}
  We present several geometric constructions of graphs with rapidly growing chromatic numbers, most of which are triangle-free or have large girth.

  \begin{enumerate}
  \setlength{\itemsep}{0pt}  
\setlength{\parsep}{0pt}
      \item We construct a triangle-free intersection graph of $n$ boxes in $\mathbb{R}^3$ with independence number $n(\log n)^{-1+o(1)}$, and thus chromatic number $(\log n)^{1-o(1)}$. This is the first improvement over the double logarithmic lower bound of Burling from 1965, and almost matches the best known upper bound $O(\log n)$. Moreover,  the bound $o(n)$ on the independence number answers a question of Walczak. 

      \item We prove that if there exists a unit distance graph in $\mathbb{R}^d$ of chromatic number $r$, then there also exists an induced unit distance graph in $\mathbb{R}^d$ of girth at least $g$ and the same chromatic number. Thus, the problem of determining the maximum chromatic number of unit distance graphs with any prescribed lower bound on the girth reduces entirely to the unrestricted problem, thereby strengthening a long line of results.

      \item We prove that there is a set of $n$ points $V\subset \mathbb{R}^2$ such that no three points are collinear, no four points are concyclic, and the integer distance graph of $V$ has chromatic number  at least $(\log n)^{1/2-o(1)}$. This gives an alternative (and considerably more elementary) solution of a recently resolved conjecture of Andr\'asfai and Erd\H{o}s.

      \item We construct an intersection graph of $n$ lines in $\mathbb{R}^3$ with girth at least $g$ and chromatic number $\Omega_g((\log n)^{1-o(1)})$. This quantitatively improves a construction of Davies.

      \item We construct a set of $n$ circles in the plane such that the tangency graph of the circles has girth at least $g$ and chromatic number $\Omega_g((\log n)^{1-o(1)})$. This quantitatively improves the result of Davies, Keller,  Kleist, Smorodinsky, and Walczak, and provides an alternative solution of Ringel's circle
problem.

      \item We construct triangle-free ordered graphs on $n$ vertices avoiding a fixed ordered path of length three as an induced subgraph, and having chromatic number $n^{\Omega(1/\log \log n)}$. This is motivated by ordered analogues of the Gy\'arf\'as-Sumner conjecture.
  \end{enumerate}


\end{abstract}

\section{Introduction}
One of the first surprising results a student might learn in a Graph Theory course is the existence of graphs which do not contain a triangle, yet have arbitrarily large chromatic number. The first constructions of families of such graphs date back to Tutte \cite{Tutte} in 1947, and independently Zykov \cite{Zykov}. Since then, many fundamentally different such families have been constructed with various methods, such as Mycielski graphs \cite{Myc}, shift graphs \cite{shift}, Burling graphs \cite{Burling}, certain Kneser graphs \cite{Kneser}, and the list goes on. Moreover, due to randomized constructions \cite{random}, we know that there exist $n$-vertex triangle-free graphs with chromatic number as large as $n^{1/2-o(1)}$. An even harder problem is to construct graphs with large \emph{girth} (length of the shortest cycle) and large chromatic number, whose existence also follows from probabilistic methods \cite{random0}.

However, it is believed that triangle-free (or large girth) graphs with large chromatic number should not appear ``naturally''. There is a large body of work studying the phenomenon that if a triangle-free graph represents geometric configurations, or satisfies certain local restrictions, then its chromatic number is small. In 1961, Berge coined the term \emph{perfect graph}, which refers to graphs where the clique number coincides with the chromatic number in every induced subgraph. In particular, triangle-free perfect graphs are 2-colorable, or equivalently, bipartite. Many extensively studied families of graphs are perfect, such as interval graphs, comparability graphs, and cographs. Later, the fundamental work of Gy\'arf\'as \cite{Gyarfas} introduced a relaxation of perfection. A family of graphs $\mathcal{G}$ is called \emph{$\chi$-bounded}, if there is a function $f:\mathbb{N}\rightarrow\mathbb{N}$ such that every member of $\mathcal{G}$ of clique number $k$ has chromatic number at most $f(k)$. Examples of $\chi$-bounded (but not perfect) families include the family of intersection graphs of axis-parallel rectangles in the plane \cite{AG60}, and the family of graphs avoiding a path of fixed length $\ell$ \cite{Gyarfas}. See the survey of Scott and Seymour \cite{SS} for a general reference.

On the other hand, there is an extent to which ``natural'' triangle-free and large girth families have small chromatic number, and there is great interest in constructing counterexamples with various additional constraints. The goal of our paper is to present several such constructions:
\begin{enumerate}
\setlength{\itemsep}{0pt}  
\setlength{\parsep}{0pt}
    \item triangle-free intersection graphs of boxes in $\mathbb{R}^3$ and higher dimensions of polylogarithmic chromatic number,
    \item induced unit distance graphs in $\mathbb{R}^d$ with arbitrarily large girth and exponential-in-$d$ chromatic number,
    \item integer distance graphs of points in strong general position of polylogarithmic chromatic numbers,
    \item intersection graphs of lines in $\mathbb{R}^3$ of arbitrarily large girth and polylogarithmic chromatic number,
    \item tangency graphs of circles of  arbitrarily large girth and polylogarithmic chromatic number,

    \item triangle-free ordered graphs avoiding a simple ordered path and having almost polynomial chromatic number.
\end{enumerate}

\bigskip
\noindent
\textbf{AI declaration.} The constructions presented in this paper were found by ChatGPT Astra after guidance from the author. We aim to present the constructions in a (hopefully) informative and human-digestible manner.

\subsection{Box-intersection graphs}

One of the starting points of the study of $\chi$-boundedness can be traced back to 1960, when Asplund and Gr\"unbaum \cite{AG60} proved that the family of intersection graphs of axis-parallel rectangles in the plane is $\chi$-bounded with $\chi$-bounding function $f(k)=O(k^2)$. Here, given a family $\mathcal{C}$ of sets, its \emph{intersection graph} is the graph whose vertices are the elements of $\mathcal{C}$, and there is an edge connecting two members if their intersection is non-empty. Demonstrating the surprising complexity of rectangle intersection graphs, it is still open whether there is a linear $\chi$-bounding function $f(k)=O(k)$, with the best known bound $f(k)=O(k\log k)$ being established by Chalermsook and Walczak \cite{CW20}.

In contrast, a celebrated construction of Burling \cite{Burling} from 1965 shows that the family of intersection graphs of boxes in $\mathbb{R}^3$ (here and later, boxes are always axis-parallel) is not $\chi$-bounded. \emph{Burling graphs} are an infinite sequence of triangle-free graphs, where the chromatic number of each $n$-vertex member is $\Theta(\log \log n)$. In his PhD thesis, Burling \cite{Burling} proved that each member of this sequence is realizable as an intersection graph of boxes in $\mathbb{R}^3$. Since then, many more applications of Burling graphs have been found \cite{PTI,PTII,PTIII}. It turns out, before one attempts to prove the $\chi$-boundedness of a family, it is advised to test whether the family contains Burling graphs.  

Thus, this settles the question whether the family of intersection graphs of boxes in $\mathbb{R}^3$ (and higher dimensions) is $\chi$-bounded. However, it remained an intriguing problem to determine how fast  the chromatic number can grow given the number of elements $n$. A simple divide-and-conquer argument shows that if $G$ is an $n$-vertex triangle-free intersection graph of boxes in $\mathbb{R}^3$, then the chromatic number satisfies $\chi(G)=O(\log n)$, see e.g. \cite{TomBox}. In contrast, Burling graphs show that one can achieve a lower bound of $\Omega(\log \log n)$, leaving an exponential gap. The lower bound has not been improved since the 1960's, but we \cite{TomSemilin} suspected for a long time  that the upper bound is closer to the truth. 

Here, we close the gap almost completely, by proving an even stronger statement. One way to show that a graph $G$ has large chromatic number $\chi(G)$ is to show that it has small independence number $\alpha(G)$, as $\chi(G)\geq n/\alpha(G)$. Walczak \cite{Walczak} asked whether there exist triangle-free intersection graphs of $n$ boxes in $\mathbb{R}^3$ with independence number $o(n)$, a question also repeated at the Oberwolfach Workshop on Discrete Geometry \cite[Problem 19]{OBF}. We answer this question by showing that there exist such graphs.

\begin{theorem}\label{thm:box_dim3}
For infinitely many $n$, there exists a triangle-free intersection graph of $n$ boxes in $\mathbb{R}^3$ with independence number $O(\frac{n\log \log n}{\log n})$, and thus chromatic number $\Omega(\frac{\log n}{\log \log n})$.
\end{theorem}

We also extend our construction to higher dimensions. We note that a divide-and-conquer argument shows that the chromatic number of a triangle-free intersection graph of $n$ boxes in $\mathbb{R}^d$ is at most $O((\log n)^{d-2})$, see e.g. \cite{TomBox}. However, the best known lower bound $\Omega(\log \log n)$ persisted in every fixed dimension. While we cannot quite match the upper bound, we present the following improvement over the lower bound.

\begin{theorem}\label{thm:box_dimd}
    For every positive integer $h$,  $d=6h-3$ and infinitely many $n$, there exists a triangle-free intersection graph of $n$ boxes in $\mathbb{R}^d$ of independence number $n(\log n)^{-h+o(1)}$, and thus chromatic number $(\log n)^{h-o(1)}$.
\end{theorem}

These constructions have further significance. Given a collection of $n$ boxes $\mathcal{B}$ in $\mathbb{R}^d$, the \emph{piercing number} (also known as \emph{transversal} number) of $\mathcal{B}$ is the minimum size of a set that intersects all members of $\mathcal{B}$, and is denoted by $\tau(\mathcal{B})$. Clearly, the piercing number is at least the independence number of $\mathcal{B}$ (referring to the independence number of the intersection graph), which we denote by $\alpha(\mathcal{B})$. While the piercing number can be strictly larger than the independence number, a celebrated conjecture of Wegner \cite{Wegner} from 1965 proposed that in dimension $d=2$, we have $\tau(\mathcal{B})\leq 2\alpha(\mathcal{B})-1$. After 60 years, this conjecture was disproved recently by Ajwani, Gajjala, Raman, Ray \cite{AGRR}, see also Bal \cite{Bal} for subsequent developments. On the other hand, it remains open whether $\tau(\mathcal{B})=O(\alpha(\mathcal{B}))$, with the best known general bound $\tau(\mathcal{B})=O(\alpha(\mathcal{B})(\log\log \alpha(\mathcal{B}))^2)$ following from a work of Correa, Feuilloley, P\'erez-Lantero, and Soto \cite{CFPS}.

In dimension $d$, the same relation $\tau(\mathcal{B})=O_d(\alpha(\mathcal{B}))$ was believed to be true \cite{Eckhoff}. However, for $d\geq 3$, we \cite{TomBox}  recently disproved this by a configuration of boxes with $\alpha(\mathcal{B})=k$ and $\tau(\mathcal{B})\geq \Omega(k(\log k/\log \log k)^{d-2})$ for every $k$. This separation between the independence and piercing numbers is also optimal up to  double logarithmic factors. The construction in \Cref{thm:box_dim3} also achieves the same separation for $d=3$, and it is fundamentally different from the construction presented in \cite{TomBox}. Indeed, if $G$ is the intersection graph of the collection of boxes $\mathcal{B}$, then $\tau(\mathcal{B})\geq |\mathcal{B}|/\omega(G)$. Therefore, if $G$ is triangle-free, we have $\tau(\mathcal{B})\geq |\mathcal{B}|/2$. Thus, \Cref{thm:box_dim3} ensures the existence of a set of $n$ boxes $\mathcal{B}$ in $\mathbb{R}^3$ such that $k:=\alpha(\mathcal{B})\leq O(n\log \log n/\log n)$ and $\tau(\mathcal{B})\geq n/2\geq \Omega(k\log k/\log \log k)$.

\subsection{Unit distance graphs}

Given a set of points $X\subset \mathbb{R}^d$, the \emph{induced unit distance graph} of $X$ is the graph whose vertices are the elements of $X$, and two vertices are joined by an edge if their distance is 1. A unit distance graph is any subgraph of an induced unit distance graph. Determining the maximum chromatic number of planar unit distance graphs is the celebrated Hadwiger-Nelson problem. The answer is known to be either $5,6,$ or $7$, where 5-chromatic unit distance graphs were first constructed not that long ago by De Grey \cite{deGrey}.

The chromatic numbers of unit distance graphs in higher dimensions are also extensively studied. A breakthrough of Frankl and Wilson \cite{FW} from 1981 showed that there are unit distance graphs in $\mathbb{R}^d$ with chromatic number growing exponentially as a function of $d$. The current record $\chi(G)\geq (1.239+o(1))^d$ is due to  Ra\u{\i}gorodski\u{\i} \cite{Raig}, while the best known upper bound $(3+o(1))^d$ is due to Larman and Rogers~\cite{LR}. 

In 1975, Erd\H{o}s \cite{ErdUnit} asked whether there exist unit distance graphs in the plane with girth at least 4 and chromatic number at least 4. In what follows, it is important to highlight the distinction between induced and non-induced unit distance graphs, as forcing large girth in the former setting turns out to be a significantly harder task.  The question of Erd\H{o}s was answered positively by Wormald \cite{Wormald}, who constructed a triangle-free 4-chromatic induced unit distance graph on more than 6000 vertices. Erd\H{o}s \cite{ErdUnit} also asked whether there exist unit distance graphs with chromatic number 4 and arbitrarily large girth. This was answered in the affirmative by O'Donnell \cite{ODonnell1,ODonnell2}. Higher dimensional analogues of this problem were first considered by Ra\u{\i}gorodski\u{\i} \cite{Raig0} and Rubanov \cite{Rubanov}. Kupavskii \cite{Kupavskii} proved that for every $g>0$ there exists $\delta=\delta(g)>0$ such that there are unit distance graphs of girth at least $g$ and chromatic number at least $(1+\delta+o(1))^d$ in $\mathbb{R}^d$; see also Alon and Kupavskii \cite{AK} for further discussion. Buci\'c and Davies \cite{BD} improved this by showing the existence of unit distance graphs in $\mathbb{R}^d$ of chromatic number at least $(1.074+o(1))^d$ and arbitrarily large fixed girth $g$. Buci\'c and Davies \cite{BD} point out that their construction does not give induced unit distance graphs, and ask whether analogous high-girth, high-chromatic constructions exist in the induced setting.

We not only answer this question, but simultaneously improve all previous results by the following simple theorem, which says that every chromatic number that can be achieved by an (induced) unit distance graph in $\mathbb{R}^d$ can also be achieved by a large girth induced unit distance graph in the same dimension.

\begin{theorem}\label{thm:main_unitdist}
    Let $d\geq 2$, and assume that there exists a unit distance graph in $\mathbb{R}^d$ of chromatic number~$r$. Then for every $g$ there exists an induced unit distance graph of chromatic number $r$ and girth at least $g$ in $\mathbb{R}^d$.
\end{theorem}

This theorem has the following simple corollary based on the previously described results.

\begin{corollary}
    Let $g$ be a fixed integer.
    \begin{enumerate}
        \item There exists an induced unit distance graph in $\mathbb{R}^d$ of girth at least $g$ and chromatic number at least $$(1.239+o(1))^d.$$
        \item  There exists an induced unit distance graph in $\mathbb{R}^2$ of girth at least $g$ and chromatic number 5. 
    \end{enumerate}
\end{corollary}

The proof of \Cref{thm:main_unitdist} is based on the following graph theoretic result, which might be of independent interest. The \emph{$k$-th Cartesian power} of a graph $H$ is the graph $H^{\square k}$, whose vertices are the $k$-tuples of vertices of $H$, and two such tuples $u$ and $v$ are joined by an edge if they differ in exactly one coordinate $i$, where $u_iv_i\in E(H)$. It was proved by Horvat and Pisanski \cite{unit_distance} that the Cartesian product of induced unit distance graphs is also an induced unit distance graph in the same dimension. See also \cite{AAS,ALS} for the same result used in Euclidean Ramsey theory. Therefore, \Cref{thm:main_unitdist} is a simple consequence of the following theorem. 

\begin{theorem}\label{thm:main_cart_prod}
     Let $H$ be a graph and $g$ be an integer. Then there exists $k=k(H,g)$ and an induced subgraph of $H^{\square k}$ of girth at least $g$ and chromatic number $\chi(H)$.
\end{theorem}

\subsection{Integer distance graphs}

Given a set of points $V$ in the plane, the \emph{integer distance graph} of $V$ is the graph on vertex set $V$ in which two points are adjacent if their distance is an integer.  A classical result of Anning and Erd\H{o}s \cite{AnErd} from 1945 states that there exist arbitrarily large point sets in the plane, not all lying on a line, whose integer distance graph is complete. On the other hand, every infinite set of points with this property is collinear. 

Say that a set of points is in \emph{strong general position} if no three points are collinear, and no four points are concyclic. Erd\H{o}s asked several times \cite{Erdos75,Erdos83,Erdos87} whether there exist  arbitrarily large sets of points in strong general position whose integer distance graph is complete. The largest known construction has 7 points \cite{KK}, while Ascher, Braune and Turchet \cite{ABT} proved that there is a uniform upper bound on the size of such sets conditioned on Lang's conjecture. See \cite{GIP} for further partial results.

In a similar spirit,  a question of Andr\'asfai and Erd\H{o}s \cite{Erdos97} asks whether there exists an infinite set of points in strong general position, whose integer distance graph has infinite chromatic number. It is easy to show that this is equivalent to asking whether there exist sets of points in strong general position, whose integer distance graph has arbitrarily large chromatic number. Very recently, the Star Fleet project \cite{Starfleet} announced a Lean-verified solution to this question. Their accompanying  report  \cite{Starfleet} describes a construction based on the Hales-Jewett theorem and rational circle tangency configurations.

Here, we present a more elementary argument, which also yields good quantitative bounds on the chromatic number. We prove that by simply randomly sampling the integer grid $\{1,\dots,m\}^2$ with a suitable probability $p$, the resulting set can be altered to yield a set in strong general position, whose integer distance graph has large chromatic number.

\begin{theorem}\label{thm:main_integer}
    For infinitely many $n$, there exists a set of $n$ points  $V\subset \mathbb{R}^2$ such that no three points are collinear, no four points are concyclic, and the integer distance graph of $V$ has independence number at most $O(\frac{n\log \log n}{(\log n)^{1/2}})$, and thus chromatic number at least $\Omega(\frac{(\log n)^{1/2}}{\log \log n})$.
\end{theorem}

\subsection{Line-intersection graphs}

A celebrated result of Pawlik, Kozik, Krawczyk, Laso\'n, Micek, Trotter, and Walczak \cite{Pawlik} shows that the family of intersection graphs of line segments in the plane is not $\chi$-bounded, thus disproving an old conjecture of Erd\H{o}s. In particular, they proved that Burling graphs can be realized as intersection graphs of line segments. Thus, there exists a collection of $n$ segments, whose intersection graph is triangle-free of chromatic number $\Omega(\log \log n)$. On the other hand, McGuinness \cite{McG} proved that the triangle-free intersection graphs of $n$ segments have chromatic number $O(\log n)$, see also \cite{FP14,RW19,Suk} for further generalizations. In contrast, if such an intersection graph has girth at least 5, then the chromatic number is bounded by an absolute constant \cite{FP14,PS}.

Pach, Tardos, and T\'oth \cite{PTT} studied the 3-dimensional version of this problem. Clearly, the family of intersection graphs of segments in $\mathbb{R}^3$ is also not $\chi$-bounded. However, in \cite{PTT}, it was asked whether the family of intersection graphs of lines in $\mathbb{R}^3$ is $\chi$-bounded. This was disproved by Norin (his argument published by Davies \cite{D21}), who showed that so-called double-shift graphs \cite{shift} can be realized as intersection graphs of lines. This provides a collection of $n$ lines, whose intersection graph is triangle-free with chromatic number $\Omega(\log\log n)$. We \cite{TomLine} constructed collections of lines that improve the bound on the chromatic number to $\Omega(n^{1/15})$, showing great contrast with the case of two-dimensional segments. Furthermore, Davies \cite{D21} proved the existence of intersection graphs of lines with arbitrarily large girth and chromatic number. His construction is based on Ramsey type results on homothetic copies of sets \cite{PV}, and thus quantitatively provides extremely poor bounds. We give an alternative construction of large girth intersection graphs of lines with polylogarithmic chromatic number.

\begin{theorem}\label{thm:main_line}
    For every positive integer $g$ and infinitely many $n$, there is an $n$-vertex intersection graph of $n$ lines in $\mathbb{R}^3$ with girth at least $g$, independence number $O_g\left(\frac{n(\log \log n)^2}{\log n}\right)$, and thus chromatic number $\Omega_g\left(\frac{\log n}{(\log \log n)^2}\right)$.
\end{theorem}

We present two constructions, as we believe both are interesting in their own right. The first construction achieves independence number $n(\log n)^{-1/2+o(1)}$, which is weaker than the promised one, but it is based on results in additive combinatorics. The second construction is more involved, and not only proves \Cref{thm:main_line}, but also can be realized as a tangency graph of circles, which we discuss in the next subsection.

\subsection{Ringel's circle problem}

A \emph{constellation} is a set of circles in the plane such that no three are tangent at the same point. The \emph{tangency graph} (or touching graph) of a set of circles is the graph whose vertices are the circles, and two are joined by an edge if they are tangent (intersect in exactly one point). Jackson and Ringel \cite{JR} asked about the maximum chromatic number of four types of constellations.

\begin{enumerate}
\setlength{\itemsep}{0pt}  
\setlength{\parsep}{0pt}
    \item Non-overlapping unit circles.
    \item Non-overlapping circles.
    \item Unit circles.
    \item Arbitrary circles.
\end{enumerate}
In cases 1 and 2, the tangency graph is planar, so the maximum chromatic number is $4$ by the celebrated Four Color theorem \cite{AH,AHK}. In case  3, the tangency graph is isomorphic to a unit distance graph, so the problem of determining its maximum chromatic number is the Hadwiger-Nelson problem, and the answer is known to be in the set $\{5,6,7\}$ \cite{deGrey}.

Whether there is a finite upper bound in case 4 was asked by Ringel \cite{Ringel} in 1959. This question, despite the significant interest,  remained open until recently, when Davies, Keller, Kleist, Smorodinsky, and Walczak \cite{DKKSW} answered it negatively. They proved the existence of tangency graphs with arbitrarily large girth and chromatic numbers. Here, we note that triangle-freeness already implies that the circles form a constellation. Their construction is based on an idea of Davies \cite{D21}, where one builds graphs with large girth and chromatic number based on the  Ramsey type results of Pr\"omel and Voigt \cite{PV} on homothetic copies of sets. A disadvantage of this approach is that the size of the resulting graph as a function of the girth and chromatic number is astronomical.

Here, we provide an alternative construction of tangency graphs of circles with large girth and chromatic number, based on parametric constructions and random sampling. A consequence of this approach is a reasonable dependence between the chromatic number and number of vertices. We remark that in our construction, all tangencies are internal, in contrast to the construction of \cite{DKKSW}, where all tangencies are external.

\begin{theorem}\label{thm:main_circle}
  For every positive integer $g$ and infinitely many $n$, there exists a collection of $n$ circles in the plane whose tangency graph has girth at least $g$, independence number $O_g\left(\frac{n(\log \log n)^2}{\log n}\right)$, and thus chromatic number $\Omega_g\left(\frac{\log n}{(\log \log n)^2}\right)$.
\end{theorem}



\subsection{Forbidding ordered trees}

A fundamental conjecture in  structural graph theory is the Gy\'arf\'as-Sumner conjecture \cite{Gy1,Sumner}, which proposes that for every fixed forest $T$, the family of graphs avoiding $T$ as an induced subgraph is $\chi$-bounded. This conjecture is known, for example, if $T$ is a path \cite{Gy1,Gyarfas}, a subdivision of a star \cite{Scott}, or a tree of radius two \cite{KP}, but it is open in general. On the other hand, Nguyen, Scott and Seymour \cite{NSS} proved that every $n$-vertex induced-$T$-free graph of clique number at most $\omega$ has chromatic number at most $\exp((\log n)^{1-c})=n^{o(1)}$ for some $c=c(T,\omega)>0$.

It is natural to wonder whether a similar result holds for ordered graphs. An \emph{ordered graph} is a graph with a linear ordering on its vertex set, and (induced) containment of ordered graphs is a containment respecting the orderings. It is easy to see that the natural ordered analogue of the Gy\'arf\'as-Sumner conjecture is false:

\medskip

\noindent
\textbf{Example.} The \emph{shift graph} $S_k$ is the graph whose vertices are the pairs $(a,b)$ with $1\leq a<b\leq k$, and two vertices $(a,b)$ and $(c,d)$ are joined by an edge if $b=c$ or $a=d$. These graphs were introduced by Erd\H{o}s and Hajnal \cite{shift}, who showed that $S_k$ is triangle-free on $n=\binom{k}{2}$ vertices of chromatic number $\Omega(\log n)$. Ordering the vertices lexicographically, the resulting ordered graph avoids every three-edge ordered path $x_1x_2x_3x_4$ with $x_1,x_3<x_2,x_4$ as an induced subgraph.

\medskip

On the other hand, it is tempting to believe that the following relaxation of the Gy\'arf\'as-Sumner conjecture continues to hold.

\begin{conjecture}
    Let $F$ be an ordered forest, $\omega\in \mathbb{N}$, and $\eps>0$, then the following holds for every sufficiently large $n$. If $G$ is an $n$-vertex ordered graph with $\omega(G)\leq \omega$ and with no induced copy of $F$, then $\chi(G)\leq n^{\eps}$.
\end{conjecture}

This conjecture is relatively unexplored. Bria\'nski, Davies, and Walczak \cite{BDW} announced a proof that if $M$ is an ordered matching, then the family of  ordered graphs avoiding $M$ as an induced subgraph is $\chi$-bounded.  Axenovich, Rollin, Ueckerdt \cite{ARU} studied the non-induced variant of the conjecture, and provided several examples of ordered forests whose exclusion does not force $\chi$-boundedness. While we do not have a solution to this conjecture yet, our final construction shows that there exist triangle-free families of ordered graphs that avoid a simple three-edge path as an induced subgraph, yet have almost polynomial chromatic number.

Say that an ordered graph is \emph{cross-path-free} if it does not contain four vertices $a<b<c<d$ such that $ac,ad,bd\in E(G)$ and $bc\not\in E(G)$. Note that being cross-path-free is more restrictive than forbidding the ordered path $cadb$, as we do not require $ab$ and $cd$ to be non-edges. See \Cref{fig:1} for an illustration.

\begin{figure}[ht]
\begin{center}
\begin{tikzpicture}

  \node[circle, draw, fill=black, inner sep=2pt,label=below:$a$] (v1) at (0,0) {};
  \node[circle, draw, fill=black, inner sep=2pt,label=below:$b$] (v2) at (2,0) {};
  \node[circle, draw, fill=black, inner sep=2pt,label=below:$c$] (v3) at (4,0) {};
  \node[circle, draw, fill=black, inner sep=2pt,label=below:$d$] (v4) at (6,0) {};
  \draw (v1) to[bend left=20] (v3);
  \draw (v1) to[bend left=30] (v4);
  \draw (v2) to[bend left=20] (v4);
  \draw[dashed] (v2) -- (v3);
\end{tikzpicture}
\end{center}
\vspace{-10pt}
\caption{The forbidden configuration.}
\label{fig:1}
\end{figure}
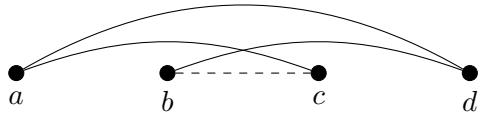

\begin{theorem}\label{thm:main_ordered}
For infinitely many $n$, there exists a triangle- and cross-path-free ordered graph on $n$ vertices with independence number $n^{1-\Omega(1/\log \log n)}$, and thus chromatic number $n^{\Omega(1/\log \log n)}$.
\end{theorem}

We highlight that our initial interest in cross-path-free ordered graphs came from a notorious combinatorial geometry problem. Given a set of points $P$ in $\mathbb{R}^2$, its \emph{box-Delaunay graph}, also known as \emph{rectangle-visibility} graph, is the graph whose vertices are the elements of $P$, and $x,y\in P$ are joined by an edge if the axis-parallel rectangle spanned by $x$ and $y$ does not contain any other point of $P$. We denote this graph by $D_{\bx}(P)$. Even, Lotker, Ron, and Smorodinsky \cite{ELRS} and Har-Peled and Smorodinsky \cite{HS} conjectured that every $n$-vertex box-Delaunay graph contains an independent set of size $\Omega(n)$. This was disproved by Chen, Pach, Szegedy, and Tardos \cite{CPST}, who showed that a random $n$-point set $P\subset [0,1]^2$ satisfies $\alpha(D_{\bx}(P))\leq O(n(\log \log n)^2/\log n)$ with high probability. Recently, Jin, Kwan and Lichev \cite{JKL} improved this by showing that  $\alpha(D_{\bx}(P))=\Theta(n(\log \log n)/\log n)$ for the random point set $P$. 

Chen, Pach, Szegedy, and Tardos \cite{CPST} proposed a natural relaxation of the now disproved conjectures of \cite{ELRS} and \cite{HS}:  we have $\alpha(D_{\bx}(P))=n^{1-o(1)}$ for every $n$-point set $P$. This seemingly innocent problem has become a well-known open problem in the area.  The best known lower bound $\alpha(D_{\bx}(P))=\Omega(n^{0.631})$ is due to Chan \cite{Chan}, improving an earlier result of \cite{AEGR}. On the other hand, we \cite{TomDel} recently showed the existence of $n$-point sets $P$ with $\alpha(D_{\bx}(P))=ne^{-\Omega(\sqrt{\log n})}$, possibly making this conjecture even more puzzling. 

A set $P\subset \mathbb{R}^2$ naturally defines the partial ordering, where  $x,y\in P$ are ordered as $x\preceq y$ if $x_1\leq y_1$ and $x_2\leq y_2$. The Hasse-diagram  of this poset, denoted by $H(P)$, is a subgraph of $D_{\bx}(P)$. It is not difficult to show that the conjecture that $\alpha(D_{\bx}(P))=n^{1-o(1)}$ holds for every $n$-point set $P$ is equivalent to the conjecture that $\alpha(H(P))=n^{1-o(1)}$ for every $n$-point set $P$. However, observe that ordering $P$ by the first coordinates of the points, the resulting ordering of $H(P)$ is a triangle-free and cross-path-free ordered graph. Therefore, a possible route to showing that $\alpha(H(P))=n^{1-o(1)}$ would be to prove that every triangle- and cross-path-free ordered graph $G$ on $n$ vertices satisfies $\alpha(G)=n^{1-o(1)}$. While \Cref{thm:main_ordered} does not invalidate this strategy, it suggests certain surprising limitations.

\bigskip

\noindent
\textbf{Organization.} In the next section, we introduce some preliminary results and notation. Then, \Cref{thm:box_dim3} is proved in \Cref{sect:box3}, followed by the proof of \Cref{thm:box_dimd} in \Cref{sect:boxd}. We prove \Cref{thm:main_unitdist} and \Cref{thm:main_cart_prod}  in \Cref{sect:unitdist}. Then, we prove \Cref{thm:main_integer} in \Cref{sect:integer}. Then, a weaker version of \Cref{thm:main_line} is proved in \Cref{sect:line}, while the proofs of \Cref{thm:main_line} and \Cref{thm:main_circle} are presented in \Cref{sect:circle}. Finally, the proof of \Cref{thm:main_ordered} is presented in \Cref{sect:ordered}.

\section{Preliminaries}
In almost all of our constructions, we employ probabilistic arguments. In particular, a common proof strategy is to take a random sample of the vertices of a certain structured graph $G$, and show that sampling destroys triangles and short cycles, while it keeps the ratio of the independence number to the number of vertices. In order to bound the independence number of random subsets of graphs, we employ the celebrated \emph{graph container method} \cite{KW82,S05}. We do not assume familiarity with this method, we present the key steps in a self-contained manner. Given a graph $G$, $\Delta(G)$ denotes its maximum degree, and if $T\subset V(G)$, then $G[T]$ is the subgraph of $G$ induced by $T$.

\bigskip

\noindent
\textbf{The graph container method.} Let $G$ be a graph. Our goal is to find a small collection of sets, called \emph{containers}, such that every independent set of $G$ is contained in one of the containers. Given an independent set $I$, we run the following algorithm to build a container. Fix an arbitrary vertex ordering $<$ of the vertices of $G$, and fix a parameter $\Delta_0$.

\begin{algorithm}
\caption{Container Algorithm}\label{alg:cont}
\begin{algorithmic}[1]
    \Require Independent set $I$
    \State Set $T:=V(G)$ and $S:=\emptyset$
    \While{$\Delta(G[T]) \geq \Delta_0$}
        \State Let $v$ be the first vertex of maximum degree in $G[T]$ with respect to $<$.
        \If{$v \notin I$}
            \State $T:=T- \{v\}$
        \Else
            \State $S:= S \cup \{v\}$
            \State $T:= T-\bigl(N_{G[T]}(v)\cup\{v\}\bigr)$
        \EndIf
    \EndWhile
    \State \Return $(S,T)$
\end{algorithmic}
\end{algorithm}

\noindent
The set $S=S(I)$ is called the \emph{fingerprint} of $I$. Standard observations about the algorithm are that 
\begin{itemize}
    \item[1.] $S\subset I\subset S\cup T$,
    \item[2.] $S$ and $T$ are disjoint,
    \item[3.] $|S|\leq |V(G)|/(\Delta_0+1)$,
    \item[4.] the set $T$ is uniquely defined by $S$,
    \item[5.] there are no edges between $S$ and $T$,
    \item[6.] $\Delta(G[T])<\Delta_0$.
\end{itemize}
  Here, 3. follows from the observation that every time the algorithm adds a vertex to $S$, it removes at least $\Delta_0+1$ vertices of $T$. By 4., we may write $T_S$ denoting the unique set $T$. In order to build a collection $\mathcal{C}$ of containers, we run the algorithm for every independent set $I$, and define
$$\mathcal{C}=\{S\cup T_S: S\text{ is a fingerprint of some independent set }I\}.$$

We also collect some standard concentration inequalities that will be used throughout our paper.

\begin{lemma}[Multiplicative Chernoff bound]
Let $X_1,\dots,X_n$ be independent Bernoulli random variables, and let $X=X_1+\dots+X_n$.
\begin{enumerate}
    \item If $t\geq 2\mathbb{E}X$, then
    $$\mathbb{P}(X\geq t)\leq \exp(-t/6).$$
    \item We have 
    $$\mathbb{P}(X\leq \mathbb{E}X/2)\leq \exp(-\mathbb{E}X/8).$$
\end{enumerate}
\end{lemma}

\begin{lemma}[Azuma-Hoeffding inequality]
Let $X_0,\dots,X_n$ be a martingale such that $|X_i-X_{i-1}|\leq c$ for $i\in [n]$. Then for every $t>0$, we have
$$\mathbb{P}(X_n-X_0>t)\leq \exp\left(-\frac{t^2}{2c^2n}\right)\quad\text{and}\quad\mathbb{P}(X_0-X_n>t)\leq \exp\left(-\frac{t^2}{2c^2n}\right).$$
\end{lemma}

\section{Box-intersection graphs in dimension 3}\label{sect:box3}

In this section, we prove \Cref{thm:box_dim3}. In particular, we deduce it from the following more general theorem.

\begin{theorem}\label{thm:box_parametric}
For every pair of positive integers $r,s$, there exists a triangle-free intersection graph of  $s^rr$ boxes in $\mathbb{R}^3$ with independence number at most $(r^2+s)s^{r-1}$.
\end{theorem}

\begin{proof}[Proof of \Cref{thm:box_dim3} given \Cref{thm:box_parametric}] Set $s=r^2$. Then there exists a triangle-free intersection graph $G$ of  $n=r^{2r+1}$ boxes in $\mathbb{R}^3$ with independence number at most $$(r^2+s)s^{r-1}=2r^{2r}=\frac{2n}{r}=O\left(\frac{n\log\log n}{\log n}\right).$$
\end{proof}

The rest of the section is devoted to the proof of \Cref{thm:box_parametric}. Let $r,s$ be positive integers, and write $Q=\{0,\dots,s-1\}^r$. Define the graph $G$ on vertex set $Q\times [r]$ such that for $(x,i),(y,j)\in Q\times [r]$ with $i<j$, 
$$(x,i)\sim (y,j)\quad\text{if }x_j<y_j\text{ and }x_{\ell}=y_{\ell}\text{ for every }\ell\in [r]\setminus \{j\}.$$

\begin{lemma}
    $G$ is triangle-free.
\end{lemma}

\begin{proof}
Assume for contradiction that $(x,i),(y,j),(z,k)\in Q\times [r]$ are the vertices of a triangle. Then $i,j,k$ are pairwise distinct, so we can assume that $i<j<k$. If $(x,i)\sim (y,j)$, then $x_j<y_j$. Moreover, $(y,j)\sim (z,k)$ implies $y_j=z_j$. But then $x_j<z_j$,  contradicting that $(x,i)\sim (z,k)$ forces $x_j=z_j$.
\end{proof}

Next, we bound the independence number of $G$. To this end, we prove a general lemma about the extremal number of certain simple multidimensional matrix patterns. We state it in a more general form than it is needed here, because its full power will come in handy later. Given a vector $x\in \mathbb{R}^n$, the \emph{sign-pattern} of $x$ is $\text{sign}(x)\in \{-1,0,1\}^n$, where $\text{sign}(x)_i$ has the same sign as $x_i$.

\begin{lemma}\label{lemma:multidim}
Let $N$ and $t$ be integers, and let $\eps\in \{-1,0,1\}^{t}$ be a sign-pattern containing $p$ non-zero entries. If $S\subset [N]^t$ such that $|S|>pN^{t-1}$, then there exist $x,y\in S$ such that $\text{sign}(x-y)=\eps$.
\end{lemma}

\begin{proof}
We prove this by induction on $p$. In case $p=0$, the statement is trivial as $S$ is non-empty and we can take any $x=y\in S$. Assume that $p\geq 1$ and let $j\in  [t]$ be an index such that $\eps_j\neq 0$. Let $\eps'\in \{-1,0,1\}^t$ be the sign-pattern  which agrees with $\eps$ on every index, except for $j$, and $\eps'_j=0$. Delete every element $z\in S$ for which there exists no $z'\in S$ with $z'_j>z_j$ and $z'_i=z_i$ for $i\in [t]\setminus \{j\}$; let $S'$ be the resulting set. We deleted at most $N^{t-1}$ elements, so $|S'|\geq |S|-N^{t-1}>(p-1)N^{t-1}$. Therefore, by our induction hypothesis, $S'$ contains two elements, $x'$ and $y'$, such that $\text{sign}(x'_i-y'_i)=\eps'_i$. In particular, $x'_j=y'_j$. If $\eps_j=1$, we can find an element $x\in S$ such that $x_j>x'_j$. But then taking $y=y'$, the pair $(x,y)$ satisfies the required  properties. Similarly, if $\eps_j=-1$, we can find $y\in S$ such that $y_j>y'_j$. But then setting $x=x'$, the pair $(x,y)$ suffices.
\end{proof}

\begin{lemma}\label{lemma:box_ind}
$\alpha(G)\leq (r^2+s)s^{r-1}.$
\end{lemma}

\begin{proof}
Let $I\subset Q\times [r]$ be an independent set of $G$. For $i\in [r]$, let $I_i=\{x\in Q: (x,i)\in I\}$. 

\begin{claim}
For $1\leq i< j\leq r$, we have $|I_i\cap I_j|\leq s^{r-1}$.
\end{claim}
\begin{proof}
    Assume that $|I_i\cap I_j|>s^{r-1}$. Define the sign-pattern $\eps\in \{-1,0,1\}^r$ such that $\eps_j=-1$, and $\eps_{\ell}=0$ for $\ell\in [r]\setminus\{j\}$. By \Cref{lemma:multidim}, we can find $x,y\in I_i\cap I_j$ such that $\text{sign}(x-y)=\eps$. But then $x_{j}<y_j$ and $x_{\ell}=y_{\ell}$ for every $\ell\in [r]\setminus\{j\}$, which is equivalent to $(x,i)\sim (y,j)$, contradiction.
\end{proof}
From this claim, we get
$$|I|=\sum_{i=1}^r |I_i|\leq \left|\bigcup_{i=1}^r I_i\right|+\sum_{1\leq i<j\leq r} |I_i\cap I_j|\leq |Q|+\binom{r}{2}s^{r-1}<s^{r}+r^2s^{r-1}.$$
\end{proof}

The most difficult part of the proof is showing that $G$ can be realized as an intersection graph of boxes in $\mathbb{R}^3$. We will prove this by explicitly defining a box for every vertex of $G$.

\begin{lemma}\label{lemma:box_realization}
$G$ can be realized as an intersection graph of boxes in $\mathbb{R}^3$.
\end{lemma}

\begin{proof}
Let $b=10s$, and let 
$$\overline{x}=\sum_{h=1}^{r}x_hb^{r-h}\quad\text{and}\quad\hat{x}_{[i]}=\sum_{h=i+1}^r x_h b^{h-1},$$
that is, $\overline{x}$ is equal to the number $x_1x_2\dots x_r$ written in base $b$, while $\hat{x}_{[i]}$ is the base-$b$ number with digits $x_r,x_{r-1},\dots,x_{i+1}$ followed by $i$ zeros. For $(x,i)\in Q\times [r]$, define the box $B_{x,i}=X_{x,i}\times Y_{x,i}\times Z_{x,i}$, where
\begin{align*}
X_{x,i}&=[2\overline{x},2\overline{x}+b^{r-i}],\\
Y_{x,i}&=[2\overline{x}-(2x_i+1)b^{r-i},2\overline{x}],\\
Z_{x,i}&=[2\hat{x}_{[i]}-2b^{i-1},2\hat{x}_{[i]}+(2x_i-1)b^{i-1}].
\end{align*}
Our goal is to prove that $B_{x,i}\cap B_{y,j}\neq \emptyset$ if and only if $(x,i)\sim (y,j)$. 

First, consider the case $x=y$. We show that $Z_{x,i}$ and $Z_{x,j}$ are disjoint, so $B_{x,i}$ and $B_{x,j}$ are also disjoint. Assume that $i<j$, then the right endpoint of $Z_{x,j}$ is $$2\hat{x}_{[j]}+(2x_j-1)b^{j-1}=2\hat{x}_{[j-1]}-b^{j-1},$$
while the left endpoint of $Z_{x,i}$ is 
$$2\hat{x}_{[i]}-2b^{i-1}>2\hat{x}_{[j-1]}-b^{j-1}.$$
Thus, $Z_{x,i}$ is completely to the right of $Z_{x,j}$.

In the rest of the argument, we assume that $x\neq y$. Let $p$ be the smallest index such that $x_p\neq y_p$, and let $q$ be the largest index such that $x_q\neq y_q$. Furthermore, without loss of generality, assume that $i\leq j$.

\begin{claim}\label{claim:boxx}
$$X_{x,i}\cap X_{y,j}\neq \emptyset \quad\Leftrightarrow\quad \begin{cases}x_p<y_p\text{ and }p>i\text{, or}\\ x_p>y_p\text{ and }p>j. \end{cases}$$
\end{claim}
\begin{proof}
We have $X_{x,i}\cap X_{y,j}\neq \emptyset$ if and only if
$$2\overline{x}\leq 2\overline{y}+b^{r-j}\quad\text{and}\quad 2\overline{y}\leq 2\overline{x}+b^{r-i}.$$ If $x_p<y_p$, then $\overline{x}<\overline{y}$, so the first inequality is satisfied, and the second inequality is satisfied if and only if $y_1=x_1,\dots,y_i=x_i$.
Similarly, if $x_p>y_p$, then $\overline{x}>\overline{y}$, so the second inequality is satisfied, and the first inequality is satisfied if $y_1=x_1,\dots,y_j=x_j$.
\end{proof}

\begin{claim}\label{claim:boxy}
$$Y_{x,i}\cap Y_{y,j}\neq \emptyset \quad\Leftrightarrow\quad \begin{cases}x_p<y_p\text{ and }p\geq j\text{, or}\\ x_p>y_p\text{ and }p\geq i. \end{cases}$$
\end{claim}

\begin{proof}
We have $Y_{x,i}\cap Y_{y,j}\neq \emptyset$ if and only if
$$2\overline{x}-(2x_i+1)b^{r-i}\leq 2\overline{y}\quad\text{and}\quad 2\overline{y}-(2y_j+1)b^{r-j}\leq 2\overline{x}.$$ If $x_p<y_p$, then $\overline{x}<\overline{y}$, so the first inequality is satisfied, and the second is satisfied if $x_1=y_1,\dots,x_{j-1}=y_{j-1}$. Similarly, if $x_p>y_p$, then $\overline{x}>\overline{y}$, so the second inequality is satisfied, and the first is satisfied if $x_1=y_1,\dots,x_{i-1}=y_{i-1}$.
\end{proof}

\begin{claim}\label{claim:boxz}
$$Z_{x,i}\cap Z_{y,j}\neq \emptyset \quad\Leftrightarrow\quad \begin{cases}i=j\text{ and }q\leq i\text{, or}\\i<j\text{ and }q\leq j\text{ and }x_j<y_j.
\end{cases}$$
\end{claim}
\begin{proof}
We have $Z_{x,i}\cap Z_{y,j}\neq \emptyset$ if and only if $$2\hat{x}_{[i]}-2b^{i-1}\leq 2\hat{y}_{[j]}+(2y_j-1)b^{j-1}\quad\text{and}\quad 2\hat{y}_{[j]}-2b^{j-1}\leq 2\hat{x}_{[i]}+(2x_i-1)b^{i-1}.$$
Consider three cases.

\medskip
\noindent
\textbf{Case 1.} $i=j$.

\noindent
If $\hat{x}_{[i]}\leq \hat{y}_{[i]}$, then the first inequality is satisfied. Also, the second inequality is then satisfied if and only if $x_r=y_r,\dots,x_{i+1}=y_{i+1}$. The same conclusion holds if $\hat{y}_{[i]}\leq \hat{x}_{[i]}$. Thus, $q\leq i$.

\medskip
\noindent
\textbf{Case 2.}  $i<j$ and $\hat{x}_{[i]}< \hat{y}_{[j]}$.

\noindent
The first inequality is satisfied. Moreover, $\hat{x}_{[i]}< \hat{y}_{[j]}$ implies that $q\geq j+1$. But then $\hat{y}_{[j]}-\hat{x}_{[i]}\geq b^{q-1}/2\geq b^{j}/2$, so the second inequality cannot be satisfied.

\medskip
\noindent
\textbf{Case 3.} $i<j$ and $\hat{y}_{[j]}\leq \hat{x}_{[i]}$.

\noindent
    The second inequality is satisfied. The first inequality is satisfied if and only if
    $x_r=y_r,\dots,x_{j+1}=y_{j+1}$ and $x_j<y_j$. Thus, $q\leq j$ and $x_j<y_j$.
\end{proof}
Assume that $B_{x,i}\cap B_{y,j}\neq \emptyset$. Then we have 
$$X_{x,i}\cap X_{y,j}\neq \emptyset\quad\text{and}\quad Y_{x,i}\cap Y_{y,j}\neq  \emptyset\quad\text{and}\quad Z_{x,i}\cap Z_{y,j}\neq \emptyset.$$
First, if $i=j$, then \Cref{claim:boxx} implies that $p>i$, and \Cref{claim:boxz} implies $q\leq i$. As $p\leq q$, this is a contradiction, so we must have $i<j$. Next, the combination of \Cref{claim:boxx} and \Cref{claim:boxy} implies that $p\geq j$. In contrast, \Cref{claim:boxz} implies $q\leq j$, so we must have $p=q=j$. Then  \Cref{claim:boxx} ensures that $x_p<y_p$. To summarize, we proved that $x$ and $y$ only differ in coordinate $p=j$, and $x_j<y_j$. This is exactly the condition which ensures that $(x,i)\sim (y,j)$ in $G$.

To finish the proof, it remains to check that if $(x,i)\sim (y,j)$ for some $i<j$, then $B_{x,i}\cap B_{y,j}\neq \emptyset$. This follows by checking that the right-hand sides of \Cref{claim:boxx}, \Cref{claim:boxy}, and \Cref{claim:boxz} are all satisfied. We have $p=j$, $x_p<y_p$ and $i<p$, so the first formula of \Cref{claim:boxx} is true. Moreover, $x_p<y_p$ and $p\geq j$, so the first formula of \Cref{claim:boxy} is true. Finally, we have $i<j$, $p=q\leq j$, and $x_j<y_j$, so the second formula of \Cref{claim:boxz} is satisfied. This finishes the proof.
\end{proof}

\noindent
\textbf{Discussion.} The above construction was found by ChatGPT Astra. While this construction might seem like it came out of nowhere (like many other brilliant constructions in mathematics), we try to illuminate how it uses ideas from known constructions of triangle-free graphs with small independence number. (So we can also claim some credit.)

Codenotti, Pudl\'ak, and Resta \cite{CPR} (see also Alon and Pudl\'ak \cite{AP01}) introduced a simple, yet powerful method to turn a $C_4$-free graph with many edges into a triangle-free graph with small independence number. Given a bipartite graph $H$, whose vertex classes are totally ordered, the \emph{superline-graph} of $H$ is the graph $G$ whose vertices are the edges of $H$, and two edges $ab,cd\in E(H)$ are joined by an edge in $G$ if $a<c$, $b<d$, and $ad\in E(H)$. If $H$ is $C_4$-free with vertex classes of sizes $n_1$ and $n_2$, and $H$ has $m$ edges, then $G$ is a triangle-free graph with $m$ vertices and independence number less than $n_1+n_2$. Alon and Pudl\'ak \cite{AP01} used the superline-graph construction to provide  explicit triangle-free graphs with small independence number. Later, we (together with Suk) \cite{SukTom21, TomSemilin} found geometric applications of superline-graph constructions. In particular, in \cite{TomSemilin}, using this approach, we constructed triangle-free semilinear graphs of polylogarithmic chromatic numbers. Here, a graph is semilinear of bounded complexity if its vertices correspond to points in a bounded dimensional real space, and its edges are defined by a bounded number of linear inequalities. In particular, box intersection graphs in $\mathbb{R}^d$ are the prototypical examples of semilinear graphs of bounded complexity. Unfortunately, where \cite{TomSemilin} fell short compared to \Cref{thm:box_dim3} was that the constructed semilinear graph is not realizable as an intersection graph of boxes.

However, the graph $G$ we constructed during the proof of \Cref{thm:box_parametric} is also a superline-graph. Indeed, let $H$ be the bipartite graph, with one vertex class $Q=\{0,\dots,s-1\}^r$, and the other vertex class $B$ composed of all the one-dimensional slices of $Q$. That is, the elements of $B$ are the sets of the form $S_{x,i}$ for $x\in Q$, $i\in [r]$, which contain the $s-1$ elements of $Q$ that differ from $x$ only in coordinate $i$, and also $x\in S_{x,i}$. The edges of $H$ are the pairs $(x,S_{x,i})$ for every $i\in [r]$. Then $H$ is clearly $C_4$-free, as any $x\neq y$ share at most one one-dimensional slice. Order the elements of $Q$ by the lexicographic ordering, and order the slices in any order respecting that $S_{x,i}<S_{y,j}$  if $i<j$.  Then the superline-graph of $H$ is exactly $G$. This provides some explanation of how $G$ could have been found. However, the part that is unclear to us is why one should expect $G$ to be realizable as an intersection graph of boxes in $\mathbb{R}^3$.

\section{Box-intersection graphs in higher dimensions}\label{sect:boxd}

In this section, we extend the previous construction to higher dimensions to prove \Cref{thm:box_dimd}. The main idea is to consider our 3-dimensional family, and take the Cartesian products of its members. After sparsifying the resulting family, we get the desired configuration of boxes. We prove the following technical result.

\begin{theorem}\label{thm:main_box2}
   Let $k$ be a positive integer and $\delta>0$, then if $r$ is a sufficiently large integer with respect to $k,\delta$, then there exists an integer $f\leq r^k$ satisfying
   $$f=f(r,k)\geq \begin{cases}\Omega_k(r^{k/2})&\text{if }k\text{ is even},\\
   r^{\lceil k/2\rceil-\delta}&\text{if }k\text{ is odd}\end{cases}$$ such that the following holds for every positive integer $s$.  There exists a triangle-free intersection graph of  $fs^{kr}$ boxes in $\mathbb{R}^{3k}$ with independence number at most $(s+kf^2)s^{rk-1}$.
\end{theorem}

Assuming this theorem, \Cref{thm:box_dimd} follows almost immediately.

\begin{proof}[Proof of \Cref{thm:box_dimd} given \Cref{thm:main_box2}]
    Set $k=2h-1$, fix some $\delta>0$, and let $r$ be sufficiently large with respect to $k,\delta$ so that $f\geq r^{h-\delta}$ exists with the requirements of \Cref{thm:main_box2}. Let $s=kf^2$ and $n=fs^{kr}$. Then there exists a triangle-free intersection graph of $n=fs^{kr}$ boxes in $\mathbb{R}^{3k}=\mathbb{R}^{6h-3}$ with independence number at most 
$$(s+kf^2)s^{rk-1}=2s^{kr}=\frac{2}{f}n.$$
Here, we have $n\leq s^{rk+1}$, so 
$$\log n\leq (rk+1)\log s=O_k(r\log r).$$
Therefore,  
$$f\geq r^{h-\delta}\geq \Omega_k\left(\left(\frac{\log n}{\log \log n}\right)^{h-\delta}\right).$$
This finishes the proof.
\end{proof}

The rest of the section is devoted to the proof of \Cref{thm:main_box2}. Let $Q$ and $G$ be defined as in the previous section, so $Q=\{0,\dots,s-1\}^r$, and $G$ is the  graph on vertex set $Q\times [r]$ such that for $(x,i),(y,j)\in Q\times [r]$ with $i<j$, 
$$(x,i)\sim (y,j)\quad\text{if }x_j<y_j\text{ and }x_{\ell}=y_{\ell}\text{ for every }\ell\in [r]\setminus \{j\}.$$ 
For $(x,i)\in Q\times [r]$, let $B_{x,i}$ be a box in $\mathbb{R}^3$ such that the intersection graph of $\{B_{x,i}:(x,i)\in Q\times [r]\}$ is $G$. A collection of such boxes exists by \Cref{lemma:box_realization}. 

 We define a family of boxes in $\mathbb{R}^{3k}$ by taking the Cartesian product of $k$-tuples of boxes $B_{x,i}$. For $(\mathbf{x},\mathbf{i})\in Q^k\times [r]^k$, define the box $B_{\mathbf{x},\mathbf{i}}$ in $\mathbb{R}^{3k}$ such that
$$B_{\mathbf{x},\mathbf{i}}=\prod_{h=1}^k B_{\mathbf{x}_h,\mathbf{i}_h}.$$
Note that the intersection graph of these boxes is the \emph{strong product} $\mathbf{G}=G^{\boxtimes k}$. Here, two vertices $(\mathbf{x},\mathbf{i}),(\mathbf{y},\mathbf{j})\in Q^k\times [r]^k$ are joined by an edge if for every $h\in [k]$, either $(\mathbf{x}_h,\mathbf{i}_h)=(\mathbf{y}_h,\mathbf{j}_h)$, or $(\mathbf{x}_h,\mathbf{i}_h)\sim (\mathbf{y}_h,\mathbf{j}_h)$ in $G$. As $G^{\boxtimes k}$ is not necessarily triangle-free,  we only keep a sparse subset of the vertices. Say that a subset $R\subset [r]^k$ is \emph{three-separating} if for every three distinct $\mathbf{a},\mathbf{b},\mathbf{c}\in R$ there is a coordinate $h\in [k]$ such that $\mathbf{a}_h,\mathbf{b}_h,\mathbf{c}_h$ are pairwise distinct.

\begin{lemma}
If $R\subset [r]^k$ is three-separating, then $Q^k\times R$ induces a triangle-free subgraph of $\mathbf{G}$.
\end{lemma}

\begin{proof}
Assume to the contrary that $(\mathbf{x},\mathbf{i}),(\mathbf{y},\mathbf{j}),(\mathbf{z},\mathbf{k})\in Q^k\times R$ are three distinct vertices forming a triangle in $\mathbf{G}$. If $\mathbf{i}=\mathbf{j}$, then $(\mathbf{x},\mathbf{i})\not\sim_{\mathbf{G}} (\mathbf{y},\mathbf{j})$. Indeed, if $h\in [k]$ is a coordinate such that $\mathbf{x}_h\neq \mathbf{y}_h$, then $(\mathbf{x}_h,\mathbf{i}_h)\not\sim (\mathbf{y}_h,\mathbf{j}_h)$ in $G$. Therefore, we may assume that $\mathbf{i},\mathbf{j},\mathbf{k}$ are pairwise distinct. But then as $R$ is three-separating, there is a coordinate $h$ such that $\mathbf{i}_h,\mathbf{j}_h,\mathbf{k}_h$ are distinct. But then $(\mathbf{x}_h,\mathbf{i}_h),(\mathbf{y}_h,\mathbf{j}_h),(\mathbf{z}_h,\mathbf{k}_h)$ are three distinct vertices of $G$, so as $G$ is triangle-free, there is no edge between at least two of them. This is a contradiction.
\end{proof}

Next, we bound the independence number of $\mathbf{G}[Q^k\times R]$. The proof of our bound proceeds similarly to the proof of \Cref{lemma:box_ind}.

\begin{lemma}
For every $R\subset [r]^k$, the independence number of $\mathbf{G}[Q^k\times R]$ is at most 
$$s^{rk}+k|R|^2s^{rk-1}.$$
\end{lemma}

\begin{proof}
     Let $I\subset Q^k\times R$ be an independent set. For $\mathbf{a}\in R$, let $$I_{\mathbf{a}}=\{\mathbf{x}\in Q^k: (\mathbf{x},\mathbf{a})\in I\}.$$
    \begin{claim}
        For distinct $\mathbf{a},\mathbf{b}\in R$,
    $$|I_{\mathbf{a}}\cap I_{\mathbf{b}}|\leq ks^{rk-1}.$$
    \end{claim}
    \begin{proof}
     Let $S=I_{\mathbf{a}}\cap I_{\mathbf{b}}\subset Q^k$, and identify $Q^k$ with $\{0,\dots,s-1\}^{rk}$. Assume to the contrary that $|S|>ks^{rk-1}$. Define the sign-pattern $\eps\in \{-1,0,1\}^{rk}$ as follows. For $g\in [rk]$, writing $g=(h-1)r+j$ with $h\in \{1,\dots,k\}$ and $j\in [1,\dots,r]$, let

    $$\eps_g=\begin{cases} -1 &\text{if }\mathbf{a}_h<\mathbf{b}_h\text{ and }j=\mathbf{b}_h,\\
    1 &\text{if }\mathbf{a}_h>\mathbf{b}_h\text{ and }j=\mathbf{a}_h,\\
    0 &\mbox{ otherwise.}\end{cases}$$
    As $\eps$ has at most $k$ non-zero entries and $|S|>ks^{rk-1}$, \Cref{lemma:multidim} ensures that there exist $\mathbf{x},\mathbf{y}\in S$ such that $\text{sign}(\mathbf{x}-\mathbf{y})=\eps$. But the sign-pattern $\eps$ is designed so that this implies $(\mathbf{x},\mathbf{a})\sim_{\mathbf{G}} (\mathbf{y},\mathbf{b})$, contradiction.
    \end{proof}
    
    From the previous claim, we have
    $$|I|=\sum_{\mathbf{a}\in R} |I_{\mathbf{a}}|\leq \left| \bigcup_{\mathbf{a}\in R} I_{\mathbf{a}}\right|+\sum_{\mathbf{a}\neq \mathbf{b}\in R} |I_{\mathbf{a}}\cap I_{\mathbf{b}}|\leq s^{rk}+|R|^2ks^{rk-1}.$$
\end{proof}

To finish the proof of \Cref{thm:box_dimd}, we need a large three-separating set $R\subset [r]^k$, so then we can set $f=|R|$. Such families are special subcases of separating hash families, and are extensively studied in combinatorics and theoretical computer science. As proved by Alon \cite{Nilli}, a probabilistic deletion argument provides a three-separating set of size $\Omega_k(r^{k/2})$ in case $k$ is even. On the other hand, if $k$ is odd, Shangguan and Ge \cite[Corollary 7.2]{ShGe} constructed three-separating sets of size $r^{\lceil k/2\rceil-o(1)}$, where  $o(1)\rightarrow 0$ as $r\rightarrow\infty$.

\section{Unit distance graphs}\label{sect:unitdist}

In this section, we prove \Cref{thm:main_unitdist} and \Cref{thm:main_cart_prod}. First, we recall the definition of the  \emph{$k$-th Cartesian power} of a graph $H$. Denoted by  $H^{\square k}$, it is the graph on vertex set $V(H)^k$, where $x,y$ are joined by an edge if $x$ and $y$ differ in a single coordinate $i\in [k]$, and $x_iy_i\in E(H)$. The advantage of Cartesian power is that if $H$ is an induced unit distance graph in $\mathbb{R}^d$, then $H^{\square k}$ is also an induced  unit distance graph in $\mathbb{R}^d$. This was proved in a more general form in \cite[Theorem 3.4]{unit_distance}, but for completeness, we also present a sketch of the proof.

\begin{lemma}\label{lemma:cartesian}
Let $H$ be an induced unit distance graph in $\mathbb{R}^d$ for $d\geq 2$. Then $H^{\square k}$ is also an induced unit distance graph in $\mathbb{R}^d$.
\end{lemma}

\begin{proof}
Let $R_1,\dots,R_k\in \text{SO}(d)$ be independent random rotation matrices. Formally, $\text{SO}(d)$ is the rotation group of $\mathbb{R}^d$, and we pick from the uniform measure, also known as Haar measure. Let $X\subset \mathbb{R}^d$ such that the unit distance graph of $X$ is isomorphic to $H$, and for $x\in X^k$, define
$$\phi(x)=\sum_{i=1}^k R_ix_i.$$
Note that if $x$ and $y$ only differ in coordinate $i$, then
$$\phi(x)-\phi(y)=R_i(x_i-y_i),$$
so $||\phi(x)-\phi(y)||=||x_i-y_i||$. On the other hand, if $x$ and $y$ differ in more than one coordinate, then $\mathbb{P}(||\phi(x)-\phi(y)||=1)=0$, and also $\mathbb{P}(\phi(x)=\phi(y))=0$ (ensuring injectivity). We omit further details.
\end{proof}

A well-known property of the Cartesian power is that $\chi(H^{\square k})=\chi(H)$ \cite{Sabidussi}.  We show that by randomly sampling a sufficiently large Cartesian power of $H$, we can retain the chromatic number while also destroying short cycles. We restate \Cref{thm:main_cart_prod} here for the reader's convenience.

\begin{theorem}\label{thm:product}
    Let $H$ be a graph and $g$ be an integer. Then there exists $k=k(H,g)$ and an induced subgraph of $H^{\square k}$ of girth at least $g$ and chromatic number $\chi(H)$.
\end{theorem}

\begin{proof}
We may assume that $H$ has at least 2 vertices, otherwise the statement is trivial. Assume that $k$ is sufficiently large with respect to $|V(H)|$ and $g$, specified later. Let $m=|V(H)|$, $r=\chi(H)$, $G=H^{\square k}$, $N:=m^{k}=|V(G)|$ and $\Delta_0=\frac{k}{rm}$. Run the Container algorithm (\ref{alg:cont}) with parameter $\Delta_0$ on $H^{\square k}$, and let 
$$\mathcal{C}=\{S\cup T_S:S\text{ is a fingerprint of some independent set }I\}$$
be the collection of containers. Note that for every $C\in\mathcal{C}$, we have $\Delta(G[C])< \Delta_0$, so 
$$e(G[C])\leq \frac{kN}{2rm}.$$
Moreover, for every fingerprint $S$, we have $|S|\leq N/\Delta_0$, which implies 
$$|\mathcal{C}|\leq \binom{N}{\leq N/\Delta_0}\leq \exp\left(\frac{2Nrm\log k}{k}\right).$$
For $i\in [k]$ and $x\in V(H)^k=V(G)$, a subset of vertices in $G$ is called the \emph{$i$-slice of $x$} if it is the set of vertices that are either $x$, or only differ from $x$ in the $i$-th coordinate. A \emph{slice} refers to an $i$-slice for any $i\in [k]$. Each slice in $H^{\square k}$ induces a copy of $H$, and there are $km^{k-1}=kN/m$ slices. 

As the chromatic number of $H$ is at least $r$, no $r-1$ independent sets can cover $V(H)$. We show that this effect in $H^{\square k}$ is greatly amplified by proving that no $r-1$ containers can cover a large portion of $G$.

\begin{claim}
    Let $C_1,\dots,C_{r-1}\in \mathcal{C}$. Then
    $$|V(G)\setminus (C_1\cup \dots\cup C_{r-1})|\geq  \frac{N}{2m}.$$
\end{claim}
\begin{proof}
Let $D=C_1\cup\dots\cup C_{r-1}$. For $i\in [k]$, let $q_i$ be the number of $i$-slices contained completely in $D$. As $H$ is not $(r-1)$-colorable, every $i$-slice contained in $D$ sends at least one edge to one of $C_1,\dots,C_{r-1}$. Therefore,
$$q_1+\dots+q_k\leq \sum_{i=1}^{r-1} e(G[C_i])\leq (r-1)\cdot \frac{kN}{2rm}<\frac{kN}{2m}.$$
On the other hand, as the $i$-slices partition $V(G)$ into $N/m$ parts of size $m$, we have
$$|D|\leq q_i m+\left(\frac{N}{m}-q_i\right)(m-1)=q_i+\frac{N(m-1)}{m}.$$
Hence, we get
$$k|D|\leq q_1+\dots+q_k+\frac{kN(m-1)}{m}\leq \frac{kN}{2m}+\frac{kN(m-1)}{m}=\frac{kN(m-1/2)}{m},$$
which implies
$$|V(G)\setminus D|\geq \frac{N}{2m}.$$
\end{proof}

Let $p=k^{-3/4}$, and let $U$ be a random $p$-sample of $V(G)$. We prove that with high probability, $G[U]$ has chromatic number at least $r$ robustly, that is, the chromatic number remains at least $r$  even after removing some vertices.   We prove this by arguing that most of $U$ cannot be covered by $r-1$ containers. Say that $U$ is \emph{$q$-robust}, if $G[U']$ has chromatic number at least $r$ for every $U'\subset U$ with $|U'|\geq |U|-q$. 

\begin{claim}
    For $q=pN/(4m)$, we have
    $$\mathbb{P}(U\text{ is }q\text{-robust})\geq \frac{2}{3}.$$
\end{claim}

\begin{proof}
 Let $C_1,\dots,C_{r-1}\in \mathcal{C}$, $D=C_1\cup\dots \cup C_{r-1}$ and let $D^c=V(G)\setminus D$. Then $|D^c|\geq N/(2m)$ by the previous claim. Write $X=|U\cap D^c|$, then $\mathbb{E}X\geq pN/2m$, and by Chernoff's inequality
\begin{align*}
\mathbb{P}(X\leq q)&\leq \mathbb{P}\left(X\leq \frac{1}{2}\mathbb{E}X\right)\leq \exp\left(-\frac{\mathbb{E}X}{8}\right)< \exp\left(-\frac{pN}{16m}\right).
\end{align*}

If $U$ is not $q$-robust, then for some $U'\subset U$ with $|U'|\geq |U|-q$, $U'$ can be covered by $r-1$ independent sets, and thus $r-1$ containers. In other words, there is an $(r-1)$-tuple of containers $C_1,\dots,C_{r-1}$ such that $|U\setminus (C_1\cup\dots\cup C_{r-1})|\leq q$. Hence, by the union bound over all $(r-1)$-tuples of containers, we get
$$\mathbb{P}(U\text{ is not }q\text{-robust})\leq |\mathcal{C}|^{r-1}\exp\left(-\frac{pN}{16m}\right)<\exp\left(\frac{2Nr^2m\log k}{k}-\frac{k^{-3/4}N}{16m}\right).$$
If $k$ is sufficiently large, then the right-hand side is at most $1/3$, finishing the proof.
\end{proof}

In what follows, we show that random sampling destroys short cycles. 

\begin{claim}
    For $\ell\geq 3$, the number of cycles in $G$ of length $\ell$ is at most 
    $$O_{m,\ell}(Nk^{\lfloor \ell/2\rfloor}).$$
\end{claim}
\begin{proof}
A cycle $v_0,\dots,v_{\ell-1}$ in $G$ can be uniquely encoded by a sequence $v_0,s_1,e_1,s_2,e_2,\dots,s_{\ell},e_{\ell}$, where  $s_i$ is the coordinate in which $v_{i-1}$ and $v_i$ differ, and $e_i$ is the corresponding edge of $H$. Every coordinate that appears in $\{s_1,\dots,s_{\ell}\}$ must appear at least twice in order to ensure that $v_0,\dots,v_{\ell-1}$ is a cycle. Thus, we have at most $\binom{k}{\leq \lfloor\ell/2\rfloor}\leq 2k^{\lfloor\ell/2\rfloor}$ choices for the set $S=\{s_1,\dots,s_{\ell}\}$. From this, we get that the number of cycles of length $\ell$ is at most
$$2Nk^{\lfloor \ell/2\rfloor}\ell^{\ell}m^{2\ell},$$
as there are $N$ choices for $v_0$, at most $2k^{\lfloor \ell/2\rfloor}$ choices for the set $S$, at most $\ell^{\ell}$ choices to arrange $s_1,\dots,s_{\ell}$ once the set $S$ is fixed, and at most $m^{2\ell}$ choices for the edges $e_1,\dots,e_{\ell}$.
\end{proof}

Let $X_{\ell}$ denote the number of cycles of length $\ell$ in $G[U]$. Then 
$$\sum_{\ell=3}^{g-1} \mathbb{E}X_{\ell}\leq \sum_{\ell=3}^{g-1} O_{m,\ell}(p^{\ell}Nk^{\lfloor \ell/2\rfloor})=\sum_{\ell=3}^{g-1} O_{m,\ell}(Nk^{\lfloor \ell/2\rfloor-3\ell/4})\leq O_{m,g}(Nk^{-1}).$$
Let $q=pN/(4m)$. If $k$ is sufficiently large with respect to $m$ and $g$, the right-hand side is at most $q/3$. But then, by Markov's inequality, 
$$\mathbb{P}\left(X_3+\dots+X_{g-1}\geq q\right)\leq \frac{1}{3}.$$
Therefore, with positive probability,  there is a choice for $U$ such that $U$ is $q$-robust, and the number of cycles in $G[U]$ of length less than $g$ is at most $q$. But then removing a vertex of each cycle, we get a subset $V\subset U$ such that $\chi(G[V])\geq r$ and $G[V]$ has girth at least $g$. 
\end{proof}

The proof of \Cref{thm:main_unitdist} follows immediately from the combination of \Cref{thm:product} and \Cref{lemma:cartesian}. Indeed, let $H_0$ be a unit distance graph of chromatic number $r$. Clearly, then there also exists an induced unit distance graph $H$ of chromatic number $r$ as well, as adding edges does not decrease the chromatic number. Then $H^{\square k}$ is also an induced unit distance graph, and every induced subgraph of $H^{\square k}$ is also an induced unit distance graph. But for sufficiently large $k$,  $H^{\square k}$ has an induced subgraph with girth at least $g$ and chromatic number $\chi(H)$, finishing the proof.

\section{Integer distance graphs}\label{sect:integer}

In this section, we prove \Cref{thm:main_integer}. For a set $V\subset \mathbb{R}^2$, let $G_V$ denote the integer distance graph of $V$.  Let $m$ be a positive integer, and consider the integer grid $Q=[m]^2$. We prove that for a suitable choice of $p$, an alteration of the random $p$-sample of $Q$ provides a subset of points $V$ in strong general position such that $G_V$ has small independence number. Indeed, it has been known (see e.g. \cite{GGK}) that sampling the grid with probability $p\approx 1/(m(\log m)^{1/2})$ yields a set of points in strong general position after small alterations, so our new contribution is showing that the resulting set also induces an integer distance graph of small independence number. In order to bound the independence number of a random sample, we employ the Container method.

As a standard component of applications of the method, we first prove a supersaturation result, showing that any large subset of $Q$ contains a vertex with large degree.

\begin{lemma}\label{lemma:integer_sat}
Let $m$ be sufficiently large, and let $U\subset [m]^2$ be of size at least $m^{1.9}$. Then
$$\Delta(G_U)\geq \Omega\left(\frac{|U|\log m}{m}\right).$$
\end{lemma}

\begin{proof}
If $(a,b)$ is an integer vector such that $a^2+b^2$ is a square (also known as \emph{Pythagorean vector}), then $[m]^2\cap \{x+t(a,b):t\in \mathbb{Z}\}$ forms a clique in $G_Q$ for every $x\in \mathbb{Z}^2$. If $1\leq a,b\leq 2m^{1/2}$ and $a,b$ are coprime, then the parallel lines $\{x+t(a,b):t\in \mathbb{Z}\}$ partition $Q$ into at most $m(a+b)$ parts. Indeed, each such line has an element in $([m]\times [b])\cup ([a]\times [m])$. Therefore, writing $L_1,\dots,L_k$ for the lines of direction $(a,b)$ intersecting $Q$, we have $k\leq m(a+b)$, and the number of edges of $G_U$ contained in one of $L_1,\dots,L_k$ is at least
$$\sum_{i=1}^k \binom{|L_i\cap U|}{2}=\frac{1}{2}\sum_{i=1}^k |L_i\cap U|^2-\frac{|U|}{2}\geq \frac{(\sum_{i=1}^k |L_i\cap U|)^2}{2k}-\frac{|U|}{2}=\frac{|U|^2}{2k}-\frac{|U|}{2}\geq \frac{|U|^2}{4m(a+b)}.$$
The first inequality is due to the Cauchy-Schwarz inequality, while in the last inequality, we used that $|U|\geq 16m^{3/2}>4m(a+b)$. Let $(u,v)$ be a pair of integers such that $1\leq u<v\leq m^{1/4}$, $u\not\equiv v\pmod{2}$ and $u,v$ are coprime. Then $(a,b)=(v^2-u^2,2uv)$ gives a unique Pythagorean vector with $a,b$ coprime and $1\leq a,b\leq 2m^{1/2}$. As lines of different directions yield different edges of $G_U$, we get
$$e(G_U)\geq \sum_{\substack{1\leq u<v\leq m^{1/4}\\ \text{gcd}(u,v)=1\\ u\not\equiv v\pmod{2}}}\frac{|U|^2}{16mv^2}.$$
By elementary number theoretic results, the sum on the right-hand side is at least $\Omega(|U|^2\log m/m)$. Indeed, for every sufficiently large $t$, there are $\Omega(t^2)$ coprime pairs $t\leq u<v\leq 2t$ with different parities, so the inequality follows by dyadic partitioning. Thus, by averaging, $G_U$ contains a vertex of degree at least $\Omega(|U|\log m/m)$.
\end{proof}

Now we are ready to employ the Container method and prove that a random $p$-sample of $Q$ has small independence number with large probability.

\begin{lemma}\label{lemma:int_ind}
Let $\eps\in (0,1/2)$, $p=\frac{\eps}{(\log m)^{1/2}m}$, and assume that $m$ is sufficiently large. Let $U$ be a random $p$-sample of $[m]^2$. Then with probability at least $2/3$, we have $|U|\geq pm^2/2$ and
$$\frac{\alpha(G_U)}{|U|}\leq O_{\eps}\left(\frac{\log\log m}{(\log m)^{1/2}}\right).$$
\end{lemma}

\begin{proof}
By the multiplicative Chernoff bound, we have $\mathbb{P}(|U|\geq pm^2/2)\geq 3/4$ for $m$ sufficiently large. 

Let $c_0$ be the constant hidden by the $\Omega(.)$ notation in \Cref{lemma:integer_sat}. Let
$$L=\frac{m^2}{c_0\log m}\quad\text{and}\quad\Delta_0=m.$$
Run the Container algorithm (\Cref{alg:cont}) on the graph $G_Q$ with parameter $\Delta_0$, and let $$\mathcal{C}=\{S\cup T_S:S\text{ is a fingerprint}\}$$ be the collection of containers.

Next, we bound the size of $S$ and $T_S$. The simple bound $|S|\leq |Q|/(\Delta_0+1)$ is not enough for our purposes, but a more delicate counting argument yields a stronger bound. 

\begin{claim}
    We have
    $$|S|\leq \frac{8m\log \log m}{c_0\log m}\quad\text{and}\quad |T_S|\leq L.$$
\end{claim}

\begin{proof}
 The Container algorithm starts with $T=V(G)$, and in each step, removes at least one vertex of $T$. In case $|T|\geq L$, we have $\Delta(G[T])\geq \frac{c_0L\log m}{m}=m=\Delta_0$ by \Cref{lemma:integer_sat}, which shows that  $|T_S|\leq L$. Moreover, \Cref{lemma:integer_sat} ensures that   $$\Delta(G[T])\geq \frac{c_0|T|\log m}{m}.$$ This means that every time the algorithm adds a vertex to the set $S$, it removes at least $\frac{c_0|T|\log m}{m}$ vertices from $T$. Hence, while the size of $T$ drops from $t$ to $t/2$ for some $t\geq 2L$, at most $\frac{2m}{c_0\log m}$ vertices have been added to $S$. This ensures that as the size of $T$ drops from $m^2$ to $L=m^2/(c_0\log m)$, at most $\frac{4m\log \log m}{c_0\log m}$  vertices have been added to $S$. On the other hand, if $|T|\leq L$, the algorithm added at most $L/\Delta_0=O(m/\log m)$ vertices to $S$. Thus, we have
 $$|S|\leq \frac{4m\log \log m}{c_0\log m}+O\left(\frac{m}{\log m}\right)\leq \frac{8m\log \log m}{c_0\log m}.$$
\end{proof}

Let $\beta=\lfloor \frac{8m\log\log m}{c_0\log m}\rfloor$ and $\alpha=4\beta$. If $\alpha(G_U)\geq \alpha$, then there exists an independent set $I\subset U$ of size at least $\alpha$, and  a fingerprint $S$ such that $S\subset I\subset T_S\cup S$. As $|S|\leq \beta$, the latter implies that 
$|T_S\cap U|\geq 3\beta$. The probability of this event can be bounded by a simple union bound as
$$\mathbb{P}(|T_S\cap U|\geq 3\beta)\leq \binom{|T_S|}{3\beta}p^{3\beta}\leq \left(\frac{
eLp}{3\beta}\right)^{3\beta}\leq \exp\left(3\beta \log \frac{Lp}{\beta} \right).$$
Hence, applying the union bound over all possible fingerprints, and noting that the events $(S\subset U)$ and $(|T_S\cap U|\geq 3\beta)$ are independent by the disjointness of $S$ and $T_S$, we get
\begin{align*}
\mathbb{P}(\alpha(G_U)\geq \alpha)&\leq \mathbb{P}(\exists S\text{ fingerprint}:\ S\subset U\text{ and }|T_S\cap U|\geq 3\beta)\\
&\leq \sum_{S\text{ fingerprint}} p^{|S|} \mathbb{P}(|T_S\cap U|\geq 3\beta)\\
&\leq \sum_{k=0}^{\beta}\binom{m^2}{k}p^k \exp\left(3\beta \log \frac{Lp}{\beta} \right)\\
&\leq \sum_{k=0}^{\beta}\exp\left(k+k\log \frac{m^2p}{k}+3\beta \log \frac{Lp}{\beta}\right).
\end{align*}
The function $x\mapsto x+x\log \frac{m^2p}{x}$ is monotone increasing for $x<\beta<m^2p/e$, so for $k\leq \beta$, the right-hand side can be further bounded as
\begin{equation}\label{eq:int1}
    (\beta+1) \exp\left(\beta+\beta\log \frac{m^2p}{\beta}+3\beta \log \frac{Lp}{\beta}\right)=(\beta+1)\exp\left(\beta+\beta\log \left[\frac{m^2p}{\beta}\cdot \left(\frac{Lp}{\beta}\right)^3\right]\right).
\end{equation}
Recalling that  $L=\Theta(m^2/\log m)$, $p=\eps/(m(\log m)^{1/2})$ and $\beta=\Theta(m\log\log m/\log m)$, we have inside the logarithm
$$\frac{m^2p}{\beta}\cdot \left(\frac{Lp}{\beta}\right)^3=\Theta_{\eps}\left(\frac{1}{(\log \log m)^4\log m}\right).$$  Therefore, the right-hand side of (\ref{eq:int1}) is tending to 0 as $m$ tends to infinity. Thus, with probability at least $2/3$, we have $\alpha(G_U)\leq \alpha$ and $|U|\geq pm^2/2$, which implies that 
$$\frac{\alpha(G_U)}{|U|}\leq O_{\eps}\left(\frac{\log \log m}{(\log m)^{1/2}}\right).$$
\end{proof}

In order to finish the proof of \Cref{thm:main_integer}, we prove that a random $p$-sample of $Q$ contains few triples of points on a line, and few 4-tuples on a circle, which can then be deleted by removing a few vertices. To this end, we use the following results about the number of collinear triples and concyclic 4-tuples in $[m]^2$. The former is a classical result of Guy and Kelly \cite{GK}, while the latter is due to a recent result of Ghosal, Goenka, and Keevash \cite{GGK}.

\begin{lemma}{\cite{GK}}
  The number of triples of points in $[m]^2$ contained on a line is $O(m^4\log m)$.
\end{lemma}

\begin{lemma}{\cite[Theorem 1.3]{GGK}}
The number of 4-tuples of points in $[m]^2$ contained on a circle is $O(m^5)$.  
\end{lemma}

Let $p=\eps m^{-1}(\log m)^{-1/2}$ for some $\eps\in (0,1/2)$ specified later, and let $U$ be a random $p$-sample of $[m]^2$. Let $X$ denote the number of collinear triples of $U$, and let $Y$ denote the number of concyclic 4-tuples of $U$. Then
$$\mathbb{E}X=O(p^3m^4\log m)=O\left(\frac{\eps^3 m}{(\log m)^{1/2}}\right)$$
and
$$\mathbb{E}Y=O(p^4m^5)=O\left(\frac{\eps^4 m}{(\log m)^2}\right).$$
Thus, by choosing $\eps>0$ sufficiently small, we have $\mathbb{E}X+\mathbb{E}Y\leq pm^2/12$. But then by Markov's inequality, we have
$$\mathbb{P}(X+Y\leq pm^2/4)\geq 2/3.$$
Thus, combining this with \Cref{lemma:int_ind}, there is a choice for $U$ such that $|U|\geq pm^2/2$, $X+Y\leq pm^2/4$ and 
$$\frac{\alpha(G_U)}{|U|}\leq O\left(\frac{\log\log m}{(\log m)^{1/2}}\right).$$
Delete an element of every collinear triple and every concyclic 4-tuple of $U$, then we get a set $V$ of  $n\geq pm^2/4$ points in general position satisfying
$$\frac{\alpha(G_V)}{|V|}\leq \frac{2\alpha(G_U)}{|U|}=O\left(\frac{\log\log n}{(\log n)^{1/2}}\right).$$
This concludes the proof of \Cref{thm:main_integer}.

\section{Line-intersection graphs}\label{sect:line}

In this section, we prove a somewhat weaker version of \Cref{thm:main_line}. We prove the following theorem.

\begin{theorem}\label{thm:line_alter}
    For every positive integer $g$ and infinitely many $n$, there is an $n$-vertex intersection graph of $n$ lines in $\mathbb{R}^3$ with girth at least $g$, independence number $O_g\left(\frac{n\log \log n}{(\log n)^{1/2}}\right)$, and thus chromatic number $\Omega_g\left(\frac{(\log n)^{1/2}}{\log \log n}\right)$.
\end{theorem}

We define a parametric collection of lines. For $(a,b)\in \mathbb{R}^2$, define the line $\ell_{a,b}$ in $\mathbb{R}^3$ such that
$$\ell_{a,b}=\{(x,ax+a^2,3a^2x+4a^3-b):x\in \mathbb{R}\}.$$
The main advantage of this family is the following intersection property.

\begin{lemma}
Let $(a,b),(A,B)\in \mathbb{R}^2$ such that $(a,b)\neq (A,B)$. Then $\ell_{a,b}$ and $\ell_{A,B}$ have a nonempty intersection if and only if 
$$(A-a)^3=(B-b).$$
\end{lemma}

\begin{proof}
Assume that 
$$(x,ax+a^2,3a^2x+4a^3-b)=(X,AX+A^2,3A^2X+4A^3-B)$$
for some real numbers $x,X\in \mathbb{R}$. Then $x=X$. If $a=A$, then we must have $b=B$ as well, thus $a\neq A$. But then the equality of the second coordinates $ax+a^2=Ax+A^2$ implies $x=-(a+A)$. Finally, the equality of the third coordinates is equivalent to
$$-3a^2(a+A)+4a^3-b=-3A^2(a+A)+4A^3-B,$$
which simplifies to
$$(A-a)^3=B-b.$$
\end{proof}

Given a set $Q\subset \mathbb{R}^2$, define the graph $G_Q$ on vertex set $Q$ such that $(a,b)\sim (A,B)$ if $(A-a)^3=B-b$. Then by the previous lemma, $G_Q$ is an intersection graph of lines in $\mathbb{R}^3$. Another corollary of this lemma is that $G_Q$ is triangle-free for every choice of $Q$.

\begin{lemma}
For every $Q\subset \mathbb{R}^2$, $G_Q$ is triangle-free.
\end{lemma}

\begin{proof}
Assume to the contrary that $(a_1,b_1),(a_2,b_2),(a_3,b_3)\in Q$ form a triangle, and write $d_1=a_2-a_1,d_2=a_3-a_2,d_3=a_1-a_3$.  Then
$$d_1+d_2+d_3=0\quad\text{and}\quad 0=(b_2-b_1)+(b_3-b_2)+(b_1-b_3)=d_1^3+d_2^3+d_3^3.$$
But this implies $d_1d_2d_3=0$, which is only possible if two of the vertices coincide. 
\end{proof}

We show that if $Q$ is an integer grid, then $G_Q$ has no large independent sets by using the following result of Lyall and Magyar \cite{LM}.

\begin{lemma}{\cite[Theorem 3]{LM}}
Let $X\subset [M]^k$ such that $(d,d^2,\dots,d^k)\not\in X-X$ for every $d\neq 0$, then
$$\frac{|X|}{M^k}\leq O_k\left(\frac{(\log \log M)^2}{\log M}\right)^{1/(k-1)}.$$
\end{lemma}

\begin{lemma}
Let $Q=[N]\times [N^3]$. Then
$$\frac{\alpha(G_Q)}{N^4}\leq O\left(\frac{\log \log N}{(\log N)^{1/2}}\right).$$
\end{lemma}

\begin{proof}
Let $I\subset Q$ be such that $\{\ell_{a,b}:(a,b)\in I\}$ is an independent set in $G_Q$. Then $(d,d^3)\not\in I-I$ for any $d\neq 0$. Define the set 
$$X=\{(a+2tN,x,b): (a,b)\in I, x\in [N^3],t\in \{0,\dots,\lfloor N^2/2\rfloor-1\}\}\subset [N^3]^3$$
then $(d,d^2,d^3)\not\in X-X$ for any $d\neq0$. Thus, the previous lemma applied with  $k=3$ and $M=N^3$ implies 
$$\frac{|I|}{N^4}\leq \frac{4|X|}{N^9}\leq O\left(\frac{\log \log N}{(\log N)^{1/2}}\right).$$
\end{proof}

Let $c>0$ be a constant such that writing 
$$f(N)=c\frac{\log \log N}{(\log N)^{1/2}},$$
we have $\alpha(G_{[N]\times [N^3]})\leq f(N)\cdot N^4$.

Our goal is to show that if $Q=[m]\times [m^3]$ for some large integer $m$, and $Q_p$ is a random $p$-sample of $Q$ for some appropriately chosen $p\in (0,1)$, then $G_{Q_p}$ has both large girth and small independence number. Here, a \emph{random $p$-sample} of a set refers to the random set we get by retaining each element of the set independently with probability $p$.

First, we show that, as long as $p$ is not too small, we have 
$$\alpha(G_{Q_p})=O(f(m) |Q_p|).$$ 
We prove this using the Graph container method (\ref{alg:cont}). We first show that every large subset of $Q$ induces many edges. This ensures that if $\Delta_0$ is set correctly, $T_S$ is small for every fingerprint $S$, otherwise $T_S$ spans too many edges.

\begin{lemma}\label{lemma:supersat}
Let $\delta>0$, then the following holds if $m$ and $C$ are sufficiently large. Let $V\subset [m]\times [m^3]$ such that $|V|\geq Cm^4 f(m)$. Then $$\Delta(G_V)>m^{1-\delta}.$$
\end{lemma}

\begin{proof}
Let us briefly outline the proof. For integers $t,u,q$, consider the map $\phi=\phi_{t,u,q}:\mathbb{R}^2\rightarrow\mathbb{R}^2$ defined as
$$\phi((x,y))=(t+qx,u+q^3y).$$
The important property of this map is that it maps edges to edges. We show that $G_V$ has many edges by considering the intersections of $V$ with  the images $\phi_{t,u,q}([N]\times [N^3])$ for an appropriately chosen $N$ and tuples of integers $(t,u,q)$. Indeed, if $|\phi_{t,u,q}([N]\times [N^3])\cap V|> f(N)N^4$, then $\phi_{t,u,q}([N]\times [N^3])\cap V$ must induce at least one edge.

 Let $\eps=\delta/10$, $N=\lfloor m^{\eps}\rfloor$ and $s=m/(4N)$. For every integer $q$ satisfying $s/2<q\leq s$, let $J_q$ be the set of pairs of integers $(t,u)$ such that $1-qN\leq t\leq m-q$ and $1-q^3N^3\leq u\leq m^3-q^3$. Then
 $$|J_q|\leq (m+sN)(m^3+s^3N^3)<4m^4.$$
 Define
$$\Gamma_{t,u,q}=\phi_{t,u,q}([N]\times [N^3])=\{(t+qx,u+q^3y):(x,y)\in [N]\times [N^3]\}.$$
Note that every element $(a,b)\in [m]\times [m^3]$ is contained in $\Gamma_{t,u,q}$ for exactly $N^4$ pairs $(t,u)\in J_q$. Indeed, for every $(x,y)\in [N]\times [N^3]$ there is a unique choice of $(t,u)\in J_q$ such that $(a,b)=(t+qx,u+q^3y)$. Thus,
$$\sum_{(t,u)\in J_q} |V\cap \Gamma_{t,u,q}|=|V|N^4\geq Cf(m)m^4N^4.$$
Say that a pair $(t,u)$ is rich if $|V\cap \Gamma_{t,u,q}|\geq Cf(m)N^4/8$. As $|J_q|\leq 4m^4$, we have
$$\sum_{\substack{(t,u)\in J_q\\ (t,u)\text{ is not rich}}} |V\cap \Gamma_{t,u,q}|\leq Cf(m)m^4N^4/2.$$
Thus, using that $|V\cap \Gamma_{t,u,q}|\leq |\Gamma_{t,u,q}|=N^4$, there are at least $Cf(m)m^4/2$ rich pairs $(t,u)$. In case $C$ is sufficiently large with respect to $\eps$, we have $Cf(m)N^4/8>f(N)N^4$, which then ensures that for every rich pair $(t,u)$, $\Gamma_{t,u,q}\cap V$ induces at least one edge in $G_V$. 

Next, we bound the number of sets $\Gamma_{t,u,q}\cap V$ with $s/2\leq q\leq s$ and $(t,u)\in J_q$ that contain a given edge $(a,b)\sim (A,B)$ in $G_V$. Assume that
$$(t+qx,u+q^3y)=(a,b)\quad\text{and}\quad (t+qX,u+q^3Y)=(A,B).$$
Then $A-a=q(X-x)$, so $q$ divides $A-a$. Thus, as $0\neq |A-a|\leq m$, there are $m^{o(1)}$ choices for $q$. If $q$ is given, we have  $t\equiv a\pmod{q}$, thus there are at most $2m/q=O(N)$ choices for $t$. Similarly, as
$u\equiv b\pmod{q^3}$,  there are at most $2m^3/q^3=O(m^3/s^3)=O(N^3)$ choices for $u$. Thus, in total, any edge may appear in $\Gamma_{t,u,q}\cap V$ for at most $m^{o(1)}N^{4}$ tuples $(t,u,q)$. 

Thus, we can lower bound the number of edges of $G_V$ by $f(m)m^{5-o(1)}N^{-5}$. Indeed, there are at least $(s/2)\cdot (Cf(m)m^4/2)=\Omega(Cf(m)m^5 N^{-1})$ tuples $(t,u,q)$ such that $V\cap \Gamma_{t,u,q}$ induces an edge, and each such edge is counted for at most $m^{o(1)}N^{4}$ different tuples. For $m$ sufficiently large, we have $f(m)m^{5-o(1)}N^{-5}>m^{5-\delta}$, which implies $\Delta(G_V)\geq m^{1-\delta}$.
\end{proof}

Now we are ready to handle the independent sets of random $p$-samples.

\begin{lemma}\label{lemma:container}
Let $\eps>0$, and let $p\geq m^{-1+\eps}$, then the following holds if $m$ is sufficiently large. If $U$ is a random $p$-sample of $[m]\times [m^3]$, then with probability at least $2/3$, we have $|U|\geq pm^4/2$ and $$\frac{\alpha(G_U)}{|U|}\leq O_{\eps}(f(m)).$$ 
\end{lemma}

\begin{proof}
Let $Q=[m]\times [m^3]$,  and $G=G_Q$.  Let $\delta=\eps/2$ and let $C$ be the constant given by \Cref{lemma:supersat}. Set $\Delta_0=m^{1-\delta}$ and $L=Cm^4f(m)$, and let $<$ be an arbitrary ordering of the vertices of $G$. Run the Container Algorithm (\Cref{alg:cont}), and let $$\mathcal{C}=\{S\cup T_S:S\text{ is a fingerprint}\}$$ be the collection of containers.

For every fingerprint $S$, we have $|S|\leq m^4/\Delta_0=m^{3+\delta}$ and $|T_S|\leq L$. A simple consequence of this is the following upper bound on the size of $\mathcal{C}$: 
$$|\mathcal{C}|\leq \binom{m^4}{\leq m^{3+\delta}}\leq \exp(2m^{3+\delta}\log m).$$
Moreover, every element of $\mathcal{C}$ has size at most $L+m^{3+\delta}\leq 2L$.

Now consider the random $p$-sample $U$. Clearly, each independent set of $G_U$ is contained in $U\cap Z$ for some container $Z\in \mathcal{C}$. Therefore, it is enough to show that $U\cap Z$ is small for every $Z\in \mathcal{C}$. As $|U\cap Z|$ is the sum of independent Bernoulli random variables and $\mathbb{E}|U\cap Z|\leq 2pL$, the multiplicative Chernoff bound gives
$$\mathbb{P}(|U\cap Z|\geq 4pL)\leq \exp(-2pL/3).$$
Thus, by the union bound
\begin{align*}
    \mathbb{P}(\alpha(G_U)\geq 4pL)&\leq \sum_{Z\in \mathcal{C}}\mathbb{P}(|U\cap Z|\geq 4pL)\leq \exp(2m^{3+\delta}\log m)\cdot\exp(-2pL/3)\\
    &\leq \exp\left(2m^{3+\delta}\log m-\frac{2C}{3}f(m)m^{3+\eps}\right).
\end{align*}
If $m$ is sufficiently large, the choice $\delta<\eps$ implies that the right-hand side is at most $1/10$.

On the other hand, by Chernoff's inequality, we have 
$$\mathbb{P}(|U|\leq pm^4/2)\leq \exp(-pm^4/8)<1/10.$$
Hence, with probability at least $2/3$, we have both $|U|\geq pm^4/2$ and $\alpha(G_U)\leq 4pL$ satisfied. But then
$$\frac{\alpha(G_U)}{|U|}\leq 8Cf(m)=O_{\eps}(f(m)),$$
finishing the proof.
\end{proof}

Next, we deal with the short cycles.

\begin{lemma}\label{lemma:girth}
    Let $k\geq 3$ be an integer and $Q=[m]\times [m^3]$. Then the number of cycles of length $k$ in $G_Q$ is at most $m^{k+2+o(1)}$.
\end{lemma}

\begin{proof}
Let $(a_1,b_1),\dots,(a_k,b_k)$ be the vertices of a cycle of length $k$, and let $d_i=a_i-a_{i-1}$ for $i\in [k]$ (where indices are meant modulo $k$). Then $b_i-b_{i-1}=d_i^3$ and  
$$d_1+\dots+d_k=0\quad\text{and}\quad d_1^3+\dots+d_k^3=0.$$
Write $D=d_3+\dots+d_k$ and $F=d_3^3+\dots+d_k^3$, then
$$d_1+d_2+D=0\quad\text{and}\quad d_1^3+d_2^3+F=0.$$
Combining the two equalities gives
$$D^3-F=3d_1d_2D.$$
First, assume that $D^3-F=0$. As $d_1,d_2\neq 0$, we must then have $D=0$. But then $d_1=-d_2$, which means $(a_2,b_2)=(a_k,b_k)$, contradiction. Thus, we may assume that $D^3-F\neq 0$. Observe that $d_1$ and $d_2$ are divisors of $D^3-F$. As $|D^3-F|\leq 2k^3m^3$, there are $m^{o(1)}$ choices for the pair $(d_1,d_2)$ given $(D,F)$.

Note that the integers $a_1,b_1,d_1,\dots,d_k$ uniquely determine the $k$-cycle. As there are at most $m^{k-2}$ choices for $d_3,\dots,d_k$, and $m^4$ choices for $(a_1,b_1)$, the total number of $k$-cycles in $G_Q$ is at most $m^{k+2+o(1)}$.
\end{proof}

Now everything is set to prove  \Cref{thm:line_alter}. Assume that $m$ is sufficiently large with respect to $g$, and let $p=m^{-1+1/g}$. Let $V$ be a $p$-sample of the elements of $[m]\times [m^3]$, then by \Cref{lemma:container}, with probability at least $2/3$, we have 
$$|V|\geq \frac{1}{2}pm^4\quad\text{and}\quad\frac{\alpha(G_V)}{|V|}\leq O_g(f(m)).$$ 
Let $X$ be the number of cycles of length less than $g$. Then by \Cref{lemma:girth}, we have
$$\mathbb{E}X\leq \sum_{k=4}^{g-1}p^k m^{k+2+o(1)}\leq \sum_{k=4}^{g-1} m^{2+o(1)+k/g}<m^3.$$
Therefore, by Markov's inequality, with probability at least $2/3$, we have $X\leq 3m^3$. Therefore, there exists an outcome of $V$ such that $|V|\geq pm^4/2=m^{3+1/g}/2$, $\alpha(G_V)/|V|\leq O_g(f(m))$, and $X<3m^3$. Given such a set $V$, remove a vertex from every cycle of length less than $g$, and let $W$ be the resulting set. Then $|W|\geq |V|-3m^3\geq |V|/2$, 
$$\frac{\alpha(G_W)}{|W|}\leq \frac{2\alpha(G_V)}{|V|}=O_{g}(f(m)),$$
and $G_W$ has girth at least $g$. Writing $n=|W|$, we have $O_{g}(f(m))=O_g(f(n))$, thus $G_W$ is an intersection graph of lines satisfying the required properties.

\bigskip

\noindent
\textbf{Discussion.} The construction presented here was partly motivated by the construction of Norin (see \cite{D21}). Indeed, he also gave a parametric family of lines, and showed that double-shift graphs can be realized as intersection graphs of lines from this family. Interestingly, this is quite different from the other constructions of lines given by us \cite{TomLine} and Davies \cite{D21}. We instructed ChatGPT Astra to look specifically for parametric families, which quickly led to the main ideas presented in this paper.

\section{Tangency graphs of circles and intersection graphs of lines}\label{sect:circle}

In this section, we prove \Cref{thm:main_line} and \Cref{thm:main_circle}. We describe a family of graphs which can be realized as both a tangency graph of circles and an intersection graph of lines. As in the previous section, we first find a parametric family, where tangency or intersection is described by a simple algebraic formula, and identify our lines or circles with points of $\mathbb{R}^3$. Then, we consider the grid $[m]^3$, and apply random sampling to destroy short cycles.

Define the function $F:\mathbb{R}^3\rightarrow \mathbb{R}$ such that
$$F(x,y,z)=xz-y^2.$$
Define the (infinite) graph $G$ on $\mathbb{R}^3$, where $u\sim v$ if $F(u-v)=0$. Also, given a finite set $Q\subset \mathbb{R}^3$, let $G_Q=G[Q].$ First, we show that $G_Q$ can be realized as both a tangency graph of circles and an intersection graph of lines, and then we can focus on the graph $G_Q$ alone.

\begin{lemma}
    Let $Q$ be a finite set.
    \begin{itemize}
        \item[1.]  $G_Q$ can be realized as a tangency graph of circles in the plane, where every tangency is internal.

        \item[2.] $G_Q$ can be realized as an intersection graph of lines in $\mathbb{R}^3$.
    \end{itemize}
    
\end{lemma}

\begin{proof}
1. Let $R>8\max_{v\in Q}||v||_{\infty}$. For $u=(a,b,c)\in \mathbb{R}^3$, define the  circle $D_{u}$ by the equation
$$(x-(a-c))^2+(y-2b)^2=(R+a+c)^2.$$
We show that $D_u$ and $D_v$ touch if and only if $F(u-v)=0$.

Write $u=(a,b,c)$ and $v=(A,B,C)$. Note that the centers of $D_{u},D_{v}$ are contained in $[-R/4,R/4]\times [-R/4,R/4]$, and their radii are at least $R/2$, so if they touch, then it is an inner tangency. The squared distance of their centers is $||p_{v}-p_{u}||^2=(A-C-(a-c))^2+(2B-2b)^2$, while the difference of the radii is $r_{v}-r_{u}=(A+C-a-c)$. Tangency occurs exactly if
$$||p_v-p_u||^2=(r_{v}-r_{u})^2,$$
which simplifies to
$$(A-a)(C-c)=(B-b)^2.$$

\medskip
2. Let $\lambda\neq 0$ be a generic real number, specified later. For $u=(a,b,c)\in \mathbb{R}^3$, define the  line 
$$\ell_{u}=\{(x,b+(a-\lambda b)x,c+(b-\lambda c)x):x\in \mathbb{R}\}.$$
For $u,v\in Q$, we show that $\ell_u$ and $\ell_v$ intersect if and only if $F(u-v)=0$. Write $u=(a,b,c)$ and $v=(A,B,C)$. The lines $\ell_u$ and $\ell_v$ intersect if and only if
$$(x,b+(a-\lambda b)x,c+(b-\lambda c)x)=(y,B+(A-\lambda B)y,C+(B-\lambda C)y)$$
has a solution. Comparing the coordinates, this is equivalent to $x=y$, 
$$(a-A-\lambda (b-B))x+b-B=0\quad\text{and}\quad (b-B-\lambda (c-C))x+c-C=0.$$
Writing $\alpha=a-A$, $\beta=b-B$ and $\gamma=c-C$, we can rewrite these two equalities as
\begin{equation}\label{equ:line}
(\alpha-\lambda \beta)x+\beta=0\quad\text{and}\quad (\beta-\lambda\gamma)x+\gamma=0.
\end{equation}

First, assume that $\alpha-\lambda \beta=0$. Then a generic choice of $\lambda$ implies  $\alpha=\beta=0$, and then $x=1/\lambda$ is a solution of (\ref{equ:line}). Therefore, $\ell_u$ and $\ell_v$ intersect, which is consistent with  $F(u-v)=F(0,0,\gamma)=0$. We can argue similarly if $\beta-\lambda \gamma=0$. 

Otherwise, there is a solution of (\ref{equ:line}) if and only if $(\alpha-\lambda \beta)\gamma=(\beta-\lambda\gamma)\beta$, which is equivalent to $F(u-v)=F(\alpha,\beta,\gamma)=0$.
\end{proof}

In what follows, our goal is to find a set $Q$ for which $G_Q$ has large girth and small independence number. In particular, we prove the following theorem, which then implies both \Cref{thm:main_line} and \Cref{thm:main_circle}.

\begin{theorem}\label{thm:mainG_Q}
For every $g$ and infinitely many $n$, there is a set $V\subset \mathbb{R}^3$ of size $n$ such that $G_V$ has girth at least $g$ and $\alpha(G_V)\leq O_g(n\frac{(\log \log n)^2}{\log n})$, and thus $\chi(G_V)\geq \Omega_g(\frac{\log n}{(\log\log n)^2})$. 
\end{theorem}

As in the proof of \Cref{thm:line_alter}, we start with the full grid $Q=[m]^3$, and apply random sampling. In order to bound the independence number of the random sample, we again apply the graph container method. Thus, we start with a supersaturation result.

\begin{lemma}\label{lemma:circle_saturation}
Let $V\subset [m]^3$ of size at least $m^{5/2}$. Then
$$\Delta(G_V)\geq \Omega\left(\frac{|V|\log m}{m^2}\right).$$
\end{lemma}

\begin{proof}
    Let $r>s$ be coprime positive integers such that $r<m^{1/4}/10$ and let $y=(r^2,rs,s^2)$. Observe that $F(ty)=0$ for every $t\in \mathbb{Z}$. The parallel lines $\{u+ty:t\in \mathbb{Z}\}$ partition $V\subset [m]^3$ into at most $3r^2m^2$ non-empty parts, each inducing a clique in $G_{V}$. Indeed, the bound on the number of lines follows from the observation that each line $\{u+ty:t\in \mathbb{Z}\}$ intersecting $[m]^3$ contains a point of $[m]^3$ whose minimal coordinate is at most $r^2$. Let $\mathcal{L}_{r,s}$ be the collection of such lines intersecting $[m]^3$, then these lines induce at least
    $$\sum_{L\in \mathcal{L}_{r,s}} \binom{|L\cap V|}{2}=\frac{1}{2}\sum_{L\in \mathcal{L}_{r,s}} (|L\cap V|^2-|L\cap V|)\geq \frac{1}{2}\left(\frac{|V|^2}{|\mathcal{L}_{r,s}|}-|V|\right)\geq \frac{|V|^2}{6r^2m^2}-\frac{|V|}{2}\geq \frac{|V|^2}{12r^2m^2}.$$
    edges in $V$, where the last inequality follows from the upper bound $r<m^{1/4}/10$. 

    In case $(r,s)\neq (r',s')$, the lines $\mathcal{L}_{r,s}$ and $\mathcal{L}_{r',s'}$ induce distinct edges in $V$, thus we get
    $$e(G_V)\geq \sum_{\substack{1\leq s<r\leq m^{1/4}/10 \\ \text{gcd}(r,s)=1}} \frac{|V|^2}{12r^2m^2}=\Omega\left(\frac{|V|^2\log m}{m^2}\right).$$
    The last equality follows from the fact that for every positive integer $t$, the number of coprime pairs of integers $(r,s)$ with $t/2<s<r<t$ is $\Omega(t^2)$.

    As $\Delta(G_V)\geq e(G_V)/|V|$, we get the desired inequality.
\end{proof}

\begin{lemma}\label{lemma:circle_independent}
Let $\eps\in (0,1/2)$, $m$ sufficiently large, and let $p=\eps/m$. Let $U$ be a random $p$-sample of $[m]^3$, then with probability at least $2/3$, we have $|U|\geq pm^3/2$ and
$$\frac{\alpha(G_U)}{|U|}\leq O_{\eps}\left(\frac{(\log \log m)^2}{\log m}\right).$$
\end{lemma}

\begin{proof}
By a simple application of the Chernoff bound, we have $\mathbb{P}(|U|\geq pm^3/2)\geq 0.9$.

Let $c_0$ be the constant hidden by the $\Omega(.)$ notation in \Cref{lemma:circle_saturation}. Let $Q=[m]^3$, $G=G_Q$, 
$$\mu=\frac{1000(\log \log m)^2}{c_0\log m},\quad L=\mu m^3\quad\text{and}\quad \Delta_0=\frac{c_0 L\log m}{m^2}=c_0\mu m\log m.$$ Let $<$ be an arbitrary ordering on the vertices of $G$. Run the Container Algorithm (\Cref{alg:cont}) for every independent set $I$ with parameter $\Delta_0$, thus producing a collection of fingerprints $S$ and accompanying sets $T_S$.

Next, we bound the size of $S$ and $T_S$. Instead of just simply using the bound $|S|\leq |Q|/(\Delta_0+1)$, we establish a stronger upper bound by more delicate counting. 

\begin{claim}
    We have
    $$|S|\leq \frac{8m^2}{c_0\log m}\log \frac{1}{\mu}=:\beta\quad\text{and}\quad |T_S|\leq L.$$
\end{claim}

\begin{proof}
 The Container algorithm starts with $T=V(G)$, and in each step, removes at least one vertex of $T$. In case $|T|\geq L$, we have $\Delta(G[T])\geq \Delta_0$, which ensures $|T_S|\leq L$. However, \Cref{lemma:circle_saturation} ensures the even stronger bound   $$\Delta(G[T])\geq \frac{c_0|T|\log m}{m^2}.$$ This means that every time the algorithm adds a vertex to the set $S$, it removes at least $\frac{c_0|T|\log m}{m^2}$ vertices from $T$. Hence, while the size of $T$ drops from $t$ to $t/2$ for some $t\geq 2L$, at most $\frac{2m^2}{c_0\log m}$ vertices have been added to $S$. This ensures that as the size of $T$ drops from $m^3$ to $L=\mu m^3$, at most  $\frac{4m^2}{c_0\log m}\log \frac{1}{\mu}$ vertices have been added to $S$. On the other hand, if $|T|\leq L$, the algorithm added at most $|T|/\Delta_0=O(m^2/\log m)$ vertices to $S$. Thus, we have
 $$|S|\leq \frac{4m^2}{c_0\log m}\log \frac{1}{\mu}+O\left(\frac{m^2}{\log m}\right)\leq \frac{8m^2}{c_0\log m}\log \frac{1}{\mu}.$$
\end{proof}

Let $\alpha=2\mu m^2$. If $\alpha(G_U)\geq \alpha$, then there exists an independent set $I\subset U$ of size at least $\alpha$, and  a fingerprint $S$ such that $S\subset I\subset T_S\cup S$. As $|S|< \mu m^2=\alpha/2$, the latter implies that 
$$|T_S\cap U|\geq \mu m^2\geq 2p L\geq 2\mathbb{E}|T_S\cap U|.$$ The probability of this event can be bounded by the Chernoff bound as
$$\mathbb{P}(|T_S\cap U|\geq \mu m^2)\leq \exp\left(-\frac{\mu m^2}{6}\right).$$
Hence, applying the union bound over all possible fingerprints, and noting that the events $(S\subset U)$ and $(|T_S\cap U|\geq \mu m^2)$ are independent by the disjointness of $S$ and $T_S$, we get
\begin{align*}
\mathbb{P}(\alpha(G_U)\geq \alpha)&\leq \mathbb{P}(\exists S\text{ fingerprint}:\ S\subset U\text{ and }|T_S\cap U|\geq \mu m^2)\\
&\leq \sum_{S\text{ fingerprint}} p^{|S|} \mathbb{P}(|T_S\cap U|\geq \mu m^2)\\
&\leq \sum_{k=0}^{\lfloor \beta\rfloor}\binom{m^3}{k}p^k \exp\left(-\frac{\mu m^2}{6}\right)\\
&\leq \sum_{k=0}^{\lfloor \beta\rfloor}\exp\left(k+k\log \frac{m^3p}{k}-\frac{\mu m^2}{6}\right).
\end{align*}
The function $x+x\log \frac{m^3p}{x}$ is monotone increasing for $x<m^3p/e=\Theta(m^2)$, so for $k\leq \beta$, the right-hand side can be further bounded as
\begin{align*}
(\beta+1) \exp\left(2\beta\log \frac{m^3p}{\beta}-\frac{\mu m^2}{6}\right)&=(\beta+1) \exp\left(\frac{16m^2}{c_0\log m}\cdot \log\frac{1}{\mu}\cdot \log \frac{c_0\eps \log m}{8\log (1/\mu)}-\frac{\mu m^2}{6}\right)\\
&\leq (\beta+1)\exp\left(\frac{16m^2(\log \log m)^2}{c_0\log m}-\frac{\mu m^2}{6}\right)
\end{align*}
The choice of $\mu$ ensures that the right-hand side is at most $0.1$ for $m$ sufficiently large. Thus, with probability at least $2/3$, we have $\alpha(G_U)\leq \alpha$ and $|U|\geq pm^3/2$, which implies that 
$$\frac{\alpha(G_U)}{|U|}\leq O_{\eps}\left(\frac{(\log \log m)^2}{\log m}\right).$$
\end{proof}

Finally, we show that sampling destroys small cycles. We highlight that if $L$ is a line of the form $\{ty+u:t\in\mathbb{R}\}$, where $y=(r^2,\pm rs,s^2)$ for some integers $(r,s)$, then $L\cap [m]^3$ induces a clique. Thus, if $G_U$ has high girth for some $U\subset [m]^3$, then we must have $|U\cap L|\leq 2$ for all such lines. We will show that this is the main bottleneck to avoid short cycles. In particular, we prove that any two points of $[m]^3$ have few common neighbors outside the line connecting them. We will make use of the following result from arithmetic number theory.

\begin{lemma}{\cite[Lemma 3.5]{VW}}\label{lemma:quadratic}
Let $A,B,C$ be non-zero integers, and let $N$ be a positive integer. Then the number of integer solutions of $$Ax^2+By^2=C$$
such that $|x|,|y|\leq N$ is at most $O_{\eps}((|ABC|N)^{\eps})$ for every $\eps>0$.
\end{lemma}

\begin{lemma}\label{lemma:common_neighbor}
Let $u,v\in [m]^3$ be distinct. 
\begin{enumerate}
    \item If $u\not\sim v$, then there are at most $m^{o(1)}$ common neighbors of $u$ and $v$.
    \item If $u\sim v$, then all common neighbors of $u$ and $v$ are on the line spanned by $u$ and $v$.
\end{enumerate} 
\end{lemma}
\begin{proof}
Let $w\in [m]^3$ be a common neighbor of $u$ and $v$. Write
$$v-u=(a,b,c),\quad w-u=(x,y,z)\quad\text{ and }\quad D=ac-b^2.$$
Note that $D=0$ is equivalent to $u\sim v$. Since $u\sim w$, we have $$xz=y^2.$$
Similarly, $v\sim w$ implies $(x-a)(z-c)=(y-b)^2$, which can be rewritten as
\begin{equation}\label{eq:7.2}
cx+az-2by=D.
\end{equation}
Multiply both sides by $x$ and substitute $xz=y^2$, then we get
$$cx^2+ay^2-2bxy=xD.$$
Assume that $a\neq 0$. Then multiplying by $a$ and completing the square, we get
$$(ac-b^2)x^2+(ay-bx)^2=axD,$$
which we can further write as $Dx^2+(ay-bx)^2=axD$. Moreover, multiplying by 4 and completing the other square, we get
\begin{equation}\label{equ:7.1}
    D(2x-a)^2+4(ay-bx)^2=Da^2.
\end{equation}
If $D\neq 0$, this is an equation of the form $AX^2+BY^2=C$, for which \Cref{lemma:quadratic} applies. Therefore, as the coefficients in this equation are bounded by $O(m^4)$, there are at most $m^{o(1)}$ solutions $(X,Y)$, which give $m^{o(1)}$ solutions $x,y$ of (\ref{equ:7.1}). Thus, $u$ and $v$ have $m^{o(1)}$ common neighbors.

If $D=0$, then (\ref{equ:7.1}) implies $ay=bx$, and (\ref{eq:7.2}) yields $cx=az$. From this,
$w-u=\frac{x}{a}(v-u),$ so the common neighbors of $u$ and $v$  lie on the line spanned by $u$ and $v$.

If $a=0$ and $c\neq 0$, then the same argument works by symmetry. Finally, if $a=c=0$, then $b\neq 0$ (as $u\neq v$). Then $D=-b^2\neq 0$ and (\ref{eq:7.2}) implies $y=b/2$.  But then $xz=y^2$ becomes $xz=b^2/4.$
The number of solutions $(x,z)$ of this equation is the number of divisors of $b^2/4$, which is also $m^{o(1)}$.
\end{proof}

\begin{lemma}\label{lemma:circle_cyclecount}
The number of triangles in $G_{[m]^3}$ is $O(m^5)$. Moreover, for $\ell\geq 4$, the number of induced cycles of length $\ell$ is at most $m^{1+\ell+o(1)}$.
\end{lemma}

\begin{proof}
Let $G=G_{[m]^3}$. First, we bound the number of triangles. By \Cref{lemma:common_neighbor}, the vertices of every triangle $u,v,w\in [m]^3$ lie on a line 
$$\{x+t(r^2,\pm rs,s^2):t\in \mathbb{Z}\},$$ where $x\in [m]^3$ and $0\leq r,s\leq \sqrt{m}$ are coprime. For fixed $(r,s)$ with $\max(r,s)=q$, there are $O(q^2m^2)$ such lines intersecting $[m]^3$, and each such line contains $O(m/q^2)$ elements. Thus, the number of triangles contributed by the union of such lines is $O(m^5/q^4)$. Thus, the total number of triangles in $G$ is at most 
$$\sum_{0\leq s\leq r\leq \sqrt{m}, r\neq 0}O\left(\frac{m^5}{r^4}\right)=O(m^5).$$

Next, we bound the number of induced $\ell$-cycles for $\ell\geq 4$. Write $\Delta$ for the maximum degree of $G$. Given $u\in [m]^{3}$ of degree $\Delta$, a vertex $v$ is a neighbor of $u$ if  $v-u=t(r^2, \pm rs,s^2)$ for some integers $r,s,t$, where $r$ and $s$ are coprime, $t\neq 0$, and $|tr^2|,|ts^2|\leq m$. Given $(r,s)$, the number of choices for $t$ is at most $\frac{2m}{\max(r^2,s^2)}$. Thus, we have
$$\Delta\leq 8\sum_{0\leq s\leq r\leq \sqrt{m}, r\neq 0}\frac{2m}{r^2}=O(m\log m).$$
Count the number of induced cycles  $v_1,\dots,v_{\ell}$ of length $\ell$ as follows. There are $m^3$ choices for $v_1$, and then if $v_{i}$ is already defined for $i=1,\dots,\ell-2$, then there are at most $\Delta$ choices for $v_{i+1}$. However, $v_{\ell}$ is a common neighbor of $v_1$ and $v_{\ell-1}$, and $v_1\not\sim v_{\ell-1}$, thus there are at most $m^{o(1)}$ choices for $v_{\ell}$. Thus, there are at most 
$$m^3\Delta^{\ell-2}m^{o(1)}=m^{\ell+1+o(1)}$$
induced cycles of length $\ell$.
\end{proof}

Now everything is set to prove \Cref{thm:mainG_Q}. Let $m$ be sufficiently large with respect to $g$, let $p=\eps/m$, and let $U$ be a random $p$-sample of $[m]^3$. By \Cref{lemma:circle_independent}, with probability at least $2/3$, we have $|U|\geq \eps m^2/2$ and 
$$\frac{\alpha(G_U)}{|U|}\leq O_{\eps}\left(\frac{(\log \log m)^2}{\log m}\right).$$
For $\ell=3,\dots,g-1$, let $X_{\ell}$ be the number of induced cycles of length $\ell$ in $G_U$. By \Cref{lemma:circle_cyclecount}, we have
$$\mathbb{E}X_{3}\leq O(m^5p^3)=O(\eps^3 m^2),$$
and for $\ell\geq 4$,
$$\mathbb{E}X_{\ell}\leq m^{\ell+1+o(1)}p^{\ell}\leq \eps^{\ell}m^{1+o(1)}.$$
Thus, choosing $\eps$ sufficiently small, we have
$$\mathbb{E}\sum_{\ell=3}^{g-1}X_{\ell}\leq O(\eps^3 m^2)<\frac{\eps m^2}{12}.$$
By Markov's inequality, with probability at least $2/3$, we thus have
$$\sum_{\ell=3}^{g-1}X_{\ell}\leq \frac{\eps m^2}{4}.$$
In conclusion, there is a choice for $U$ which simultaneously satisfies
$$|U|\geq \eps m^2/2\quad\text{and}\quad\sum_{\ell=3}^{g-1}X_{\ell}\leq \frac{\eps m^2}{4}\quad\text{and}\quad\frac{\alpha(G_U)}{|U|}\leq O\left(\frac{(\log \log m)^2}{\log m}\right).$$
But then we can remove a vertex from every induced cycle of length at most $g-1$ to get a set $V\subset U$ of size at least $|U|/2\geq \eps m^2/4$ such that $G_V$ has girth at least $g$. Moreover, we have
$$\frac{\alpha(G_V)}{|V|}\leq \frac{2\alpha(G_U)}{|U|}=O\left(\frac{(\log \log m)^2}{\log m}\right).$$
Setting $n=|V|$, $G_V$ has girth at least $g$ and satisfies $\alpha(G_V)\leq O\left(\frac{(\log \log n)^2}{\log n}n\right)$, finishing the proof of \Cref{thm:mainG_Q}.

\section{Forbidden induced ordered pattern}\label{sect:ordered}

In this section, we prove \Cref{thm:main_ordered}. Let $K$ and $L$ be two triangle- and cross-path-free ordered graphs. We prove that by a certain randomized graph operation, we can create a new triangle- and cross-path-free ordered graph by combining $K$ and $L$. 

Given an integer $b$, let $K^{(b)}$ denote the $b$ blow-up of $K$. Formally, we replace every vertex $v\in V(K)$ by an independent interval $I_v$ of size $b$ such that if $v<w$ are two vertices of $K$, then $I_v<I_w$. Moreover, if $vw\in E(K)$, then we put a complete bipartite graph between $I_v$ and $I_w$, otherwise, there is no edge between the two intervals. 

\begin{lemma}
 $K^{(b)}$ is triangle- and cross-path-free.
 \end{lemma}

 \begin{proof}
 It is easy to see that $K^{(b)}$ is triangle-free. Therefore, assume to the contrary that it is not cross-path-free, then there exists $u_1<u_2<u_3<u_4\in V(K^{(b)})$ such that $u_1u_3,u_1u_4,u_2u_4\in E(K^{(b)})$, but $u_2u_3\not\in E(K^{(b)})$. Assume that $u_i$ is in the blow-up interval of the vertex $v_i\in V(K)$. As the neighborhoods of $u_1$ and $u_2$ are different, we have $v_1<v_2$. Similarly, $v_3<v_4$. Also, we have $v_1v_3,v_1v_4,v_2v_4\in E(K)$. If $v_2=v_3$, then $K$ contains a triangle, contradiction. If $v_2<v_3$, then $u_2u_3\not\in E(K^{(b)})$ implies that $v_2v_3\not\in E(K)$, so $v_1,v_2,v_3,v_4$ contradict that $K$ is cross-path-free.
 \end{proof}

Let $k=|V(K)|$ and $\ell=|V(L)|$, write
$$\rho_K=\frac{k}{\alpha(K)}\quad\text{and}\quad\rho_L=\frac{\ell}{\alpha(L)},$$
and fix some positive integers $a$ and $b$.

Define the random graph $G=G(K,L,a,b)$ as follows. The vertex set of $G$ is divided into $\ell$ intervals $J_w$ of size $abk$, indexed by the vertices $w$ of $L$. Moreover, each $J_w$ induces $a$ disjoint copies of $K^{(b)}$ in $G$, denoted by $K_{w,1},\dots,K_{w,a}$, such that each copy occupies an interval. Let $\tau:V(G)\rightarrow [kb]$ be a labeling of the vertices of $G$ such that for every $w\in V(L)$ and $p\in [a]$, $\tau|_{V(K_{w,p})}$ is a random permutation of $[kb]$, chosen independently for every $(w,p)$. Given two vertices $u,v\in V(G)$ such that $u\in J_w$ and $v\in J_z$, connect $u$ and $v$ by an edge if $\tau(u)=\tau(v)$ and $wz$ is an edge of $L$. The resulting graph is $G$. Write 
$$n=abk\ell$$
for the number of vertices of $G$. See \Cref{fig:2} for an illustration of $G$.

\begin{figure}[ht]
\begin{center}
\begin{tikzpicture}[
    x=3.2mm,
    y=3.2mm,
    vertex/.style={circle,fill=black,inner sep=1.4pt},
    downedge/.style={thin,opacity=.7},
    every label/.style={font=\scriptsize,inner sep=1pt}
]


\foreach \i/\x/\lab in {
1/0/4,2/1/1,3/2/6,4/3/2,5/4/5,6/5/3,
7/7/2,8/8/5,9/9/1,10/10/6,11/11/3,12/12/4,
13/16/6,14/17/3,15/18/5,16/19/1,17/20/4,18/21/2,
19/23/1,20/24/4,21/25/2,22/26/5,23/27/6,24/28/3,
25/32/5,26/33/2,27/34/4,28/35/3,29/36/1,30/37/6,
31/39/3,32/40/6,33/41/1,34/42/4,35/43/2,36/44/5
}{
\node[vertex,label=below:{\lab}] (v\i) at (\x,0) {};
}

\draw (0,-1.5) -- (0,-2) -- (12,-2) -- (12,-1.5);
\node at (6,-3) {$J_1$};

\draw (16,-1.5) -- (16,-2) -- (28,-2) -- (28,-1.5);
\node at (22,-3) {$J_2$};

\draw (32,-1.5) -- (32,-2) -- (44,-2) -- (44,-1.5);
\node at (38,-3) {$J_3$};

\node at (2.5,2) {$K_{1,1}$};
\node at (9.5,2) {$K_{1,2}$};

\node at (18.5,2) {$K_{2,1}$};
\node at (25.5,2) {$K_{2,2}$};

\node at (34.5,2) {$K_{3,1}$};
\node at (41.5,2) {$K_{3,2}$};

\draw (v1) to[bend left=40] (v3);
\draw (v1) to[bend left=40] (v4);
\draw (v2) to[bend left=10] (v3);
\draw (v2) to[bend left=40] (v4);

\draw (v3) to[bend left=40] (v5);
\draw (v3) to[bend left=40] (v6);
\draw (v4) to[bend left=10] (v5);
\draw (v4) to[bend left=40] (v6);

\draw (v7) to[bend left=40] (v9);
\draw (v7) to[bend left=40] (v10);
\draw (v8) to[bend left=10] (v9);
\draw (v8) to[bend left=40] (v10);

\draw (v9) to[bend left=40] (v11);
\draw (v9) to[bend left=40] (v12);
\draw (v10) to[bend left=10] (v11);
\draw (v10) to[bend left=40] (v12);

\draw (v13) to[bend left=40] (v15);
\draw (v13) to[bend left=40] (v16);
\draw (v14) to[bend left=10] (v15);
\draw (v14) to[bend left=40] (v16);

\draw (v15) to[bend left=40] (v17);
\draw (v15) to[bend left=40] (v18);
\draw (v16) to[bend left=10] (v17);
\draw (v16) to[bend left=40] (v18);

\draw (v19) to[bend left=40] (v21);
\draw (v19) to[bend left=40] (v22);
\draw (v20) to[bend left=10] (v21);
\draw (v20) to[bend left=40] (v22);

\draw (v21) to[bend left=40] (v23);
\draw (v21) to[bend left=40] (v24);
\draw (v22) to[bend left=10] (v23);
\draw (v22) to[bend left=40] (v24);

\draw (v25) to[bend left=40] (v27);
\draw (v25) to[bend left=40] (v28);
\draw (v26) to[bend left=10] (v27);
\draw (v26) to[bend left=40] (v28);

\draw (v27) to[bend left=40] (v29);
\draw (v27) to[bend left=40] (v30);
\draw (v28) to[bend left=10] (v29);
\draw (v28) to[bend left=40] (v30);

\draw (v31) to[bend left=40] (v33);
\draw (v31) to[bend left=40] (v34);
\draw (v32) to[bend left=10] (v33);
\draw (v32) to[bend left=40] (v34);

\draw (v33) to[bend left=40] (v35);
\draw (v33) to[bend left=40] (v36);
\draw (v34) to[bend left=10] (v35);
\draw (v34) to[bend left=40] (v36);

\end{tikzpicture}

\end{center}
\vspace{-10pt}
\caption{An illustration of $G$, where $a=b=2$, and $K=L$ are the 2-edge ordered path. Each interval of length six is a copy of $K^{(2)}$, and the labels denote the random assignment $\tau$. There are further edges connecting vertices of the same label between the pairs $(J_1,J_2)$ and $(J_2,J_3)$, which are not drawn for clarity.}
\label{fig:2}
\end{figure}
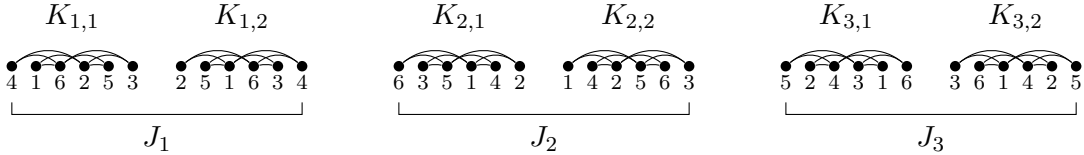

\begin{lemma}
$G$ is triangle- and cross-path-free.
\end{lemma}

\begin{proof}
First, we show that $G$ is triangle-free. Assume to the contrary that $u_1<u_2<u_3$ are vertices of a triangle. Then $u_1,u_2,u_3$ cannot be contained in the same interval $J_w$ as $K^{(b)}$ is triangle-free. Also, if two of them are contained in the same interval, but not the third, then $\tau(u_1)=\tau(u_2)=\tau(u_3)$. However, two vertices in the same interval and assigned the same number are not connected by an edge. Finally, it cannot be that all three vertices are contained in different intervals as $L$ is triangle-free.

Next, we prove that $G$ is cross-path-free. Assume to the contrary that $u_1<u_2<u_3<u_4$ are vertices of $G$ such that $u_1u_3,u_1u_4,u_2u_4$ are edges, but $u_2u_3$ is not an edge. Assume that $u_i$ is contained in the interval $J_{w_i}$ for $i\in [4]$, then $w_1\leq w_2\leq w_3\leq w_4$. We analyze 8 different cases depending on the equalities between $w_1,w_2,w_3,w_4$.
\begin{description}
    \item[$w_1=w_2=w_3=w_4$:] we get a contradiction as $K^{(b)}$ is cross-path-free.
    \item[$w_1<w_2=w_3=w_4$:] $u_1u_3$ and $u_1u_4$ are edges across intervals, which forces $\tau(u_1)=\tau(u_3)=\tau(u_4)$. However, as $u_2<u_3<u_4$ are contained in $J_{w_2}$, this forces $u_3$ and $u_4$ to be in different copies  of $K^{(b)}$, and thus $u_2$ and $u_4$ are also in different copies. But then there is no edge between $u_2$ and $u_4$, contradiction.
    \item[$w_1=w_2<w_3=w_4$:]  We have $\tau(u_1)=\tau(u_2)=\tau(u_3)=\tau(u_4)$, and $w_1w_3\in E(L)$, which implies that $u_2u_3$ is also an edge of $G$, contradiction. 
    \item[$w_1=w_2=w_3<w_4$:] analogous to the case $w_1<w_2=w_3=w_4$.
    \item[$w_1<w_2<w_3=w_4$:] as $u_1u_4$, $u_2u_4$, $u_1u_3$ are edges across intervals, we must have $\tau(u_1)=\tau(u_2)=\tau(u_3)=\tau(u_4)$. Moreover $w_1w_3\in E(L)$ and $w_2w_3=w_2w_4\in E(L)$, which forces $u_2u_3$ to be an edge, contradiction.
    \item[$w_1<w_2=w_3<w_4$:] we get a contradiction as $L$ is triangle-free.
    \item[$w_1=w_2<w_3<w_4$:] analogous to the case $w_1<w_2<w_3=w_4$.
    \item[$w_1<w_2<w_3<w_4$:] we get a contradiction as $L$ is cross-path-free.
\end{description}
\end{proof}

\begin{lemma}
Assume that 
$$a\geq \rho_K\rho_L^2\quad\text{and}\quad b\geq 2\rho_K^2\rho_L^2.$$
Then with positive probability,
$$\alpha(G)\leq \frac{4n}{\rho_K\rho_L}.$$
\end{lemma}

\begin{proof}
Let $\mathcal{I}=(I_{w,p})_{w\in V(L),p\in [a]}$ be a tuple, where $I_{w,p}$ is an independent set in $K_{w,p}$ such that if some vertex is contained in $I_{w,p}$, then all its blow-up twins are also contained in $I_{w,p}$. Define  $S(\mathcal{I})=\bigcup_{w\in V(L), p\in [a]} I_{w,p}$. Note that there are at most $2^{k}$ choices for the set $I_{w,p}$ for every fixed $(w,p)$, so the number of tuples $\mathcal{I}$ can be upper bounded as 
$$2^{ak\ell}=2^{n/b}.$$ Clearly, every independent set of $G$ is contained in some $S(\mathcal{I})$.

Next, we reveal our random assignment $\tau$, and thus the random edges across the intervals $(J_w)_{w\in V(L)}$. Fix some tuple $\mathcal{I}$ and set $S=S(\mathcal{I})$. For $w\in V(L)$ and $i\in [kb]$, let $N_{w,i}$ be the number of vertices of $S(\mathcal{I})\cap J_w$ labeled $i$. Define
$$Z_i=\max_{\substack{I\subset V(L)\\I\text{ is independent}}} \sum_{w\in I} N_{w,i}$$
and let $Z=Z_1+\dots+Z_{kb}$. Then $Z$ is an upper bound on the maximum size of an independent set of $G[S]$. Next, we bound the expectation of $Z$. We have
\begin{equation}\label{equ:order1}
    \mathbb{E}N_{w,i}\leq \frac{\alpha(K)}{k}a=:t.
\end{equation}
Indeed,  $|S\cap V(K_{w,p})|\leq b\alpha(K)$, and each vertex of $S\cap V(K_{w,p})$ receives the label $i$ with probability $1/(bk)$.  Moreover, as $N_{w,i}$ is a sum of independent indicator random variables, we also have
$$\Var(N_{w,i})\leq \mathbb{E}N_{w,i}\leq  t.$$
For every independent set $I$ of $L$, 
$$\sum_{w\in I}N_{w,i}\leq \sum_{w\in I}(t+|N_{w,i}-\mathbb{E}N_{w,i}|)\leq t\alpha(L) + \sum_{w\in I} |N_{w,i}-\mathbb{E}N_{w,i}|\leq t\alpha(L) + \sum_{w\in V(L)} |N_{w,i}-\mathbb{E}N_{w,i}|.$$
Thus, we also have
\begin{equation}\label{equ:order2}
Z_i\leq t\alpha(L)+\sum_{w\in V(L)} |N_{w,i}-\mathbb{E}N_{w,i}|.
\end{equation}
Here,
$$\mathbb{E}|N_{w,i}-\mathbb{E}N_{w,i}|\leq \sqrt{\Var(N_{w,i})}\le \sqrt{t}.$$
Taking the expectation of both sides of (\ref{equ:order2}), we get
$$\mathbb{E}Z_i\leq \alpha(L)t+\ell\sqrt{t}.$$
By our lower bound $a\geq (k/\alpha(K))\cdot(\ell/\alpha(L))^2$, the right-hand side is at most $2\alpha(L)t$. Thus, we conclude
$$\mathbb{E}Z\leq 2kbt\alpha(L)=2ab\cdot\alpha(K)\alpha(L)=\frac{2n}{\rho_K\rho_L}.$$

Next, we argue that $Z$ is concentrated around its mean. For $j=0,\dots,n$, define the martingale
$$M_j=\mathbb{E}(Z|\tau\text{ is revealed on the first }j\text{ vertices of }G).$$
The main observation is that $|M_j-M_{j+1}|\leq 2$, because for every $(w,p)$, the assignment $\tau|_{V(K_{w,p})}$ is an independent random permutation of $[kb]$, and swapping any two elements of the permutation changes the value of $Z$ by at most $2$. Hence, we can apply the Azuma-Hoeffding inequality to get that
$$\mathbb{P}\left(Z-\mathbb{E}Z\geq \frac{2n}{\rho_K\rho_L}\right)\leq \exp\left(-\frac{n}{2\rho_K^2\rho_L^2}\right).$$
Thus, the probability that $S$ contains an independent set of size at least $4n/(\rho_K\rho_L)$ is at most $\exp(-n/(2\rho_K^2\rho_L^2))$. However, recall that the number of choices for $S$ is at most $2^{n/b}$, thus, by the union bound,
$$\mathbb{P}\left(\alpha(G)\geq \frac{4n}{\rho_K\rho_L}\right)\leq 2^{n/b}\exp\left(-\frac{n}{2\rho_K^2\rho_L^2}\right)$$
Our lower bound $b\geq 2\rho_K^2\rho_L^2$ ensures that the right-hand side is strictly less than 1, finishing the proof.
\end{proof}

Let $G_0$ be a triangle- and cross-path-free ordered graph such that $\rho_0=|V(G_0)|/\alpha(G_0)>8$. Such a graph exists: as discussed in the introduction, if $P$ is a planar point set, then the Hasse-diagram of the poset induced by $P$ is triangle- and cross-path-free, assuming the points are ordered by their first coordinates. As proved in \cite{CPST,JKL,TomDel}, there are such graphs $G_0$ with $|V(G_0)|/\alpha(G_0)$ arbitrarily large.

Assuming $G_i$ is already defined with $n_i$ vertices and $\rho_i=\frac{n_i}{\alpha(G_i)}$, let $a_i=\lceil \rho_i^3\rceil $ and $b_i=\lceil 2\rho_i^4\rceil$. Let $G_{i+1}$ be an outcome of the random graph $G(G_i,G_i,a_i,b_i)$ for which $$\alpha(G_{i+1})\leq \frac{4|V(G_{i+1})|}{\rho_i^2}.$$
Then we have the recursion
$$n_{i+1}\leq 6\rho_i^7 n_i^2\quad\text{and}\quad\rho_{i+1}\geq \frac{\rho_i^2}{4}.$$
Moreover, by induction, $G_{i+1}$ is a triangle- and cross-path-free ordered graph.
By our assumption $\rho_0>8$, we get an infinite sequence of graphs $G_0,G_1,\dots$ with $\rho_0<\rho_1<\dots$. 
Let 
$$A_i=\log_2 n_i\quad\text{and}\quad B_i=\log_2 (\rho_i/4).$$
Since  $\rho_0>8$, we have $B_0>0$ and 
$$B_{i+1}\geq 2B_i\quad\text{and}\quad A_{i+1}\leq 2A_i+7B_i+17.$$
Iterating these inequalities, for $j<i$, we get
$$B_{j}\leq 2^{j-i}B_i\quad\text{and}\quad A_i\leq 2^i (A_0+17)+7\sum_{j=0}^{i-1}2^{i-j-1}B_{j}.$$
Thus,
$$A_i\leq \left(\frac{A_0+17}{B_0}+\frac{7i}{2}\right)B_i.$$
Finally, as $2^iB_0\leq B_i\leq A_i$, we have $i\leq \log_2 A_i=\log_2\log_2 n_i$. Thus, the previously highlighted inequality shows that 
$$\rho_i=4\cdot 2^{B_i}\geq 2^{\Omega(\log n_i/\log\log n_i)},$$ finishing the proof of \Cref{thm:main_ordered}.



\section*{Statements and Declarations}
\vspace{-10pt}
No financial or non-financial interests are directly or indirectly related to the work submitted for publication.
\vspace{-10pt}
\section*{Data availability}
\vspace{-10pt}
There is no data associated with the work.
\vspace{-10pt}
\section*{Acknowledgments}
\vspace{-10pt}
IT acknowledges the support of the Swedish Research Council grant VR 2023-03375. We are grateful to James Davies, who pointed out the problem of Erd\H{o}s \cite{Erdos97} on integer distance graphs, and to Bartosz Walczak for valuable discussions.

\end{document}